\documentclass[11pt]{article}

\usepackage{amsmath,amssymb,amsfonts,textcomp,amsthm,xifthen,psfrag,graphicx,color,MnSymbol}
\usepackage[T1]{fontenc}
\usepackage{hyperref}
\usepackage{placeins}
\usepackage{xargs}
\usepackage{pgfplots}
\usepackage{pgfplotstable}
\pgfplotsset{compat=newest}
\usepgfplotslibrary{groupplots}
\usepgfplotslibrary{external}
\usepackage{cancel}
\usepackage[normalem]{ulem}  

\definecolor{navyblue}{rgb}{0.0, 0.0, 0.5}

\date{}

\newtheorem{theorem}{Theorem}
\newtheorem{lemma}[theorem]{Lemma}
\newtheorem{cor}[theorem]{Corollary}
\newtheorem{prop}[theorem]{Proposition}
\newtheorem{remark}[theorem]{Remark}
\theoremstyle{definition}

\newcommand{\discard}[1]{}

\newcommand{\tnorm}[1]{|\!|\!| #1 |\!|\!|}

\newcommand{\wilde}{\widetilde}
\newcommand{\R}{\ensuremath{\mathbb{R}}}

\newcommand{\dual}[2]{\langle#1\hspace*{.5mm},#2\rangle}
\newcommand{\vdual}[2]{\langle#1\hspace*{.5mm},#2\rangle}
\newcommand{\ds}{\,\mathrm{d}s}
\newcommand{\dy}{\,\mathrm{d}y}

\newcommand{\HG}{\boldsymbol{H}\!G}
\newcommand{\Geps}[1]{{\Gamma_\epsilon(#1)}}

\DeclareMathOperator{\id}{\boldsymbol{I}}

\newcommand{\Vpot}{\widetilde V}
\newcommand{\Kpot}{\widetilde K_\nu}
\newcommand{\Npot}{\mathcal{G}}
\newcommand{\NIO}{\boldsymbol{\mathcal{G}}}
\newcommand{\No}{\mathcal{G}_0}
\newcommand{\Nt}{\mathcal{G}_n}
\newcommand{\V}{\boldsymbol{V}}
\newcommand{\K}{\boldsymbol{K}_\nu}
\newcommand{\Vo}{V_0}
\newcommand{\Vt}{V_n}
\newcommand{\Ko}{K_{\nu,0}}
\newcommand{\Kt}{K_{\nu,n}}
\newcommand{\Kteps}{K_{\nu,n,\epsilon}}
\newcommand{\Keps}{K_{\nu,\epsilon}}
\newcommand{\npi}[3][]{\pi_{#3,#2}^{#1}}

\newcommand{\grad}{\nabla}
\newcommand{\Grad}{\boldsymbol{\nabla}}
\newcommand{\Gradgrad}{\Grad\grad}

\let\div\undefined
\DeclareMathOperator{\div}{div}
\DeclareMathOperator{\Div}{Div}
\DeclareMathOperator{\divDiv}{divDiv}
\newcommand{\dDiv}{\mathrm{dDiv}}

\newcommand{\trGg}[1]{\mathrm{tr}^{\mathrm{Ggrad}}(#1)}
\newcommand{\trGge}[1]{\mathrm{tr}^{\mathrm{Ggrad}}_e(#1)}
\newcommand{\trdD}[1]{\mathrm{tr}^{\mathrm{dDiv}}(#1)}
\newcommand{\trdDe}[1]{\mathrm{tr}^{\mathrm{dDiv}}_e(#1)}
\newcommand{\trGgwo}{\mathrm{tr}^{\mathrm{Ggrad}}}
\newcommand{\trdDwo}{\mathrm{tr}^{\mathrm{dDiv}}}
\newcommand{\nDiv}[1]{\mathrm{nDiv_{eff}}(#1)}

\newcommand{\tr}[1]{\boldsymbol{#1}}
\newcommand{\trD}[1]{{#1}_0}
\newcommand{\trhD}[1]{{#1}_{h,0}}
\newcommand{\trN}[1]{{#1}_n}
\newcommand{\trhN}[1]{{#1}_{h,n}}
\newcommand{\trnn}[1]{{#1}_\mathit{nn}}
\newcommand{\trhnn}[1]{{#1}_\mathit{h,nn}}
\newcommand{\trsf}[1]{{#1}_\mathit{sf}}
\newcommand{\trhsf}[1]{{#1}_\mathit{h,sf}}
\newcommand{\trJ}[1]{{#1}_{J}}
\newcommand{\trhJ}[1]{{#1}_{h,J}}

\newcommand{\jGg}[1]{[#1]_{\mathrm{Ggrad}}}
\newcommand{\jdD}[1]{[#1]_{\mathrm{dDiv}}}
\newcommand{\avGg}[1]{\left\{#1\right\}_{\mathrm{Ggrad}}}

\newcommand{\bM}{\boldsymbol{M}}
\newcommand{\bQ}{\boldsymbol{Q}}

\newcommand{\cN}{\mathcal{N}}
\newcommand{\cE}{\mathcal{E}}
\newcommand{\cC}{\boldsymbol{\mathcal{C}}}

\let\SS\undefined
\newcommand{\SS}{\mathbb{S}}

\title{Boundary elements for clamped Kirchhoff--Love plates
\thanks{Supported by ANID-Chile through FONDECYT projects 1230013, 1250070, and
        the Deutsche Forschungsgemeinschaft (DFG, German Research Foundation)
        under Germany's Excellence Strategy -- EXC-2047/1 -- 390685813.}}
\author{
Thomas F\"uhrer\thanks{
Facultad de Matem\'aticas, Pontificia Universidad Cat\'olica de Chile,
Avenida Vicu\~na Mackenna 4860, Santiago, Chile,
email: {\tt thfuhrer@uc.cl}}
\and
Gregor Gantner\thanks{
Institut f\"ur Numerische Simulation, Universit\"at Bonn, 
Friedrich-Hirzebruch-Allee 7,
53115 Bonn, Germany 
\& Department of Applied Mathematics,  
University of Twente, 
P.O. Box 217, 
7500 AE Enschede, The Netherlands
email: {\tt gregor.gantner@utwente.nl}}
\and
Norbert Heuer\thanks{
Facultad de Matem\'aticas, Pontificia Universidad Cat\'olica de Chile,
Avenida Vicu\~na Mackenna 4860, Santiago, Chile,
email: {\tt nheuer@uc.cl}}}

\begin{document}
\maketitle

\begin{abstract}
We present a Galerkin boundary element method for clamped Kirchhoff--Love plates
with piecewise smooth boundary. It is a direct method based on the representation
formula and requires the inversion of the single-layer operator, an application of
the double-layer operator to the Dirichlet data, and, in the presence of a vertical load, an application of the Dirichlet trace of the Newton potential to that load.
We present trace approximation spaces of arbitrary order, required for both the Dirichlet
data and the unknown Neumann trace. Our boundary element method is quasi-optimal with respect to the natural trace norm
 and achieves optimal convergence order under minimal regularity assumptions.
We provide explicit representations of all three integral operators 
and discuss the implementation of the appearing integrals.
Numerical experiments for smooth and non-smooth domains confirm predicted convergence rates.



\end{abstract}

\section{Introduction} \label{sec_intro}

We establish and analyze a boundary element method for the Dirichlet problem
of the biharmonic equation with possibly non-zero right-hand side function.
The biharmonic operator is particularly relevant for problems from solid and fluid mechanics.
The biharmonic equation itself
appears in the stream function formulation of 2D Stokes flow,
in plane stress analysis with the Airy function,
and in the Kirchhoff--Love model for the bending and vibration of thin plates,
see, e.g., \cite{HappelB_65_LRN,Muskhelishvili_53_SBP,TimoshenkoWK_59_TPS}.
In this paper, we focus on the latter model. But our results are generally applicable
for the solution of the biharmonic equation in two dimensions.
Before going into the details, let us place our findings in their historical context.

For a long time, the Galerkin boundary element method has been an established tool for the
numerical solution of boundary value problems of
second order with constant coefficients and homogeneous right-hand side.
Its mathematical analysis started in the 1970s, see, e.g.,
the works by Hsiao, N\'ed\'elec, Stephan, Wendland, and co-authors
\cite{HsiaoMC_73_SBV,NedelecP_73_MVE,HsiaoW_77_FEM,GiroireN_78_NSE,WendlandSH_79_IEM}.
In the 1980s, Costabel and Stephan made notable steps towards the analysis on
non-smooth boundaries, see \cite{CostabelS_83_NDD,CostabelS_85_BIE,Costabel_88_BIO},
specifically for boundary integral equations of the first kind.
There are now standard references for the analysis of boundary integral operators
on Lipschitz domains,
\cite{Costabel_88_BIO,McLean_00_SES},
and the boundary element method,
\cite{GwinnerS_18_ABE,HsiaoW_21_BIE,SauterS_11_BEM,Steinbach_08_NAM}.

There have been attempts to extend the theoretical tools developed for Laplace-type problems to
the bi-Laplacian, or biharmonic, operator. So far, the results have been limited.
In the following, we mention some of them. 

Christiansen and Hougard \cite{ChristiansenH_78_IPI} provide a boundary integral representation
for the Dirichlet problem on smooth boundaries and prove its solvability,
where a related open conjecture is shown by Fuglede~\cite{fuglede81}.
Giroire and N\'ed\'elec \cite{GiroireN_95_NSB} develop a formulation for the Neumann
problem and, essential for numerical approximations, provide an integration-by-parts
technique to transform dualities of hypersingular operators into weakly singular ones,
again for smooth curves.
Hsiao and Wendland \cite{HsiaoW_21_BIE} study boundary integral operators for the
bi-Laplacian in the case of smooth domains.
Costabel and Dauge \cite{CostabelD_96_IBS} consider the single-layer operator
of the bi-Laplacian on Lipschitz curves and prove its invertibility
for appropriately scaled domains. (Notably, they also present some numerical experiments
for the related eigenvalue problem with lowest-order approximation.)
For piecewise smooth curves, Schmidt and Khoromskij \cite{SchmidtK_99_BIE,Schmidt_01_BIO}
develop the family of boundary integral operators that constitute the so-called Calder\'on projector 
and study their mapping properties. For a critical result on trace relations,
they refer to Jakovlev \cite{Jakovlev_61_BPF}.
In \cite{CakoniHW_05_BIE}, Cakoni {\it et al.} analyze systems of boundary integral
equations for the bi-Laplacian with mixed boundary conditions.
They assume that the boundary is locally Lipschitz differentiable.
Mitrea and Mitrea \cite{MitreaM_13_BVP} analyze layer potentials of the bi-Laplacian
in any space dimension, within the framework of a class of non-Hilbertian spaces.

Achievable approximation orders depend on the regularity of solutions and possible
singularities. For elliptic problems of second order, this is a well-studied subject.
Literature is scarce in the case of the bi-Laplacian.
Blum and Rannacher \cite{BlumR_80_BVP} analyze these corner singularities
in the case of combinations of different boundary conditions;
for finite element analyses involving corner singularities, see, e.g.,
\cite{Stephan_79_CMF,DeCosterNS_15_SBD}.

We are only aware of three contributions to the mathematical analysis of boundary elements
for the bi-Laplacian. Costabel {\it et al.} \cite{CostabelSW_83_BIE} use Mellin techniques
to analyze a boundary integral equation and its approximation on polygons. 
Furthermore, there is a preliminary analysis of Galerkin approximations on polygons
by Bourlard \cite{Bourlard_88_PDB}, which however does not take corner forces at vertices (see \eqref{trdD_comp}) into account 
and hinges on appropriately graded meshes.
Moreover, Costabel and Saranen \cite{CostabelS_89_BEA} analyze boundary element approximations
for the Dirichlet problem, again for smooth domains.
We remark that there is a considerable number of publications from the engineering community
on boundary elements for plate problems. We only mention some early contributions,
e.g., \cite{Bezine_78_BIF,morris79,GuoShuM_86_BEM,HartmannZ_86_DBE,pb89,Beskos_91_BEA}.

In this paper, we establish a boundary element method for the Dirichlet problem
of the biharmonic operator on bounded simply connected domains $\Omega$ with piecewise smooth boundary $\Gamma$.
Our method is based on the representation formula with fundamental solution
and the application of a Dirichlet trace operator.
In the absence of a load, this procedure gives rise to a boundary integral equation
of the first kind and involves the single- and double-layer operators $\V$ and $\K$, respectively,
which relate the Cauchy data (Dirichlet and Neumann traces) of the unknown solution
(relation \eqref{BIE} in Section~\ref{sec_BIO}). 
While the boundary element method is particularly advantageous in this setting, 
engineers have also successfully applied the method in the presence of loads;
see the references mentioned before. In that case,
the integral equation additionally involves the Dirichlet trace $\NIO$ of the Newton potential $\Npot$, 
whose discretization requires meshing $\Omega$.

Differently from previous research, a critical ingredient is a proper analysis of the
involved trace operators and spaces, not only permitting duality relations between them,
but also being sufficiently explicit to allow for discretizations with local
basis ``functions'', aka boundary elements. These trace operators and spaces have been
developed in \cite{FuehrerHN_19_UFK} in the context of the discontinuous Petrov--Galerkin
method with optimal test functions. Dirichlet traces are generated by a trace
operator acting on $H^2(\Omega)$, denoted as $\trGgwo$ with image $H^{3/2,1/2}(\Gamma)$.
It represents the canonical trace of functions and their normal derivative.
The Neumann traces stem from $\trdDwo$ with image $H^{-3/2,-1/2}(\Gamma)$,
an operator acting on tensors of bending moments
that live in $H(\dDiv,\Omega;\SS)$
(symmetric $L_2$ tensors with twice-iterated divergence in $L_2$).
It represents three components that are well-known in the mechanics community:
the normal bending moments, the effective shear forces, and corner forces at vertices,
cf., e.g., \cite{VentselK_01_TPS}.
As mentioned before, the traces from $\trGgwo$ and $\trdDwo$ are in duality. In fact,
they are expressions of the second Green identity, which is at the heart of
\emph{direct} representations by boundary integral equations.
We derive the boundary integral equation as in \cite{SchmidtK_99_BIE},
the only difference being the novel representation of trace operators and spaces from
\cite{FuehrerHN_19_UFK}.

Our variational formulation (\eqref{BIE_saddle} in Section~\ref{sec_BEM})
requires a Lagrangian multiplier (a linear polynomial)
to account for the possible non-invertibility of the single-layer operator $\V$. 
The well-posedness (Proposition~\ref{prop_wellposed} in Section~\ref{sec_BEM_VF})
follows from the coercivity of $\V$ on the corresponding quotient
space, proved in \cite{CostabelD_96_IBS}.

We then proceed to develop explicit representations of both boundary integral operators
(Theorems~\ref{thm_BIO_V},~\ref{thm_BIO_K} in Section~\ref{sec_BIO}).
By nature of the trace operators, there is a specific contribution from
each component of the traces. This is a critical result for numerical implementations.
Since we consider a Dirichlet problem, the solution of our boundary integral equation
is a Neumann trace living in $H^{-3/2,-1/2}(\Gamma)$. There is a \emph{low}-order discretization
for this space in \cite{FuehrerHN_19_UFK}, with an error estimate for bending moments
$\bM\in H^2(\Omega)^{2\times 2}$, improved to a reduced regularity of
$\bM\in H^1(\Omega)^{2\times 2}$ with $\div\Div\bM\in L_2(\Omega)^2$ by
\cite[Proposition 11]{FuehrerHH_23_DMQ}.
Here, we provide approximation spaces of any degree for $H^{-3/2,1/2}(\Gamma)$
(Section~\ref{sec_BEM_hp}).
The \emph{lowest}-order case has fewer degrees of freedom than the element in \cite{FuehrerHN_19_UFK}.
Our general error estimate (Theorem~\ref{thm_BestApp} in Section~\ref{sec_BEM_est_neu})
is based on a regularity with local characterization on the boundary and is of optimal order.

A numerical implementation of our boundary element method
(\eqref{BIE_saddle:disc} in Section~\ref{sec_BEM_conv}) requires to approximate
the Dirichlet data, i.e., we also need discretization spaces for $H^{3/2,1/2}(\Gamma)$
with corresponding estimates of approximation errors. Again, we define discrete
Dirichlet trace spaces of any order (Section~\ref{sec_BEM_hp}) and prove
optimal approximation orders (Theorem~\ref{thm_NodInt} in Section~\ref{sec_BEM_est_dir}).
Our boundary element method is quasi-optimal and achieves optimal order
(Theorem~\ref{thm_BEM} and Corollary~\ref{cor_BEM} in Section~\ref{sec_BEM_conv}).

\subsection{Outline}

We have just discussed the main contributions of this paper. Its remaining structure is
as follows. In the next section, we briefly introduce the Kirchhoff--Love plate bending model
and its representation by the biharmonic equation. All the details on spaces and
traces are reviewed in Section~\ref{sec_spaces}. Section~\ref{sec_BIO} is devoted
to the single- and double-layer integral operators, the Dirichlet trace of the Newton potential, and their explicit representations.
Section~\ref{sec_BEM} develops the variational formulation
of the model problem (in~Section~\ref{sec_BEM_VF}), establishes the trace approximation spaces
(in Section~\ref{sec_BEM_hp}), provides the approximation error estimates in $H^{3/2,1/2}(\Gamma)$ (in Section~\ref{sec_BEM_est_dir}) and in $H^{-3/2,-1/2}(\Gamma)$ (in Section~\ref{sec_BEM_est_neu}),
defines the boundary element method and proves its convergence (in Section~\ref{sec_BEM_conv}).
In Section~\ref{sec_num}, we present several numerical experiments, for
different smooth and non-smooth domains and the three lowest-order approximations.
They confirm the predicted convergence orders.
Appendix~\ref{sec_app} provides details on the calculation of all the appearing integrals.

\subsection{Notation}

Throughout this paper, the notation $A\lesssim B$ means that there is a positive constant
$c$, independent of involved functions and geometric dimensions, such that $A\le cB$.
The relation $A\eqsim B$ indicates that $A\lesssim B$ and $B\lesssim A$.

\section{Model problem} \label{sec_model}

Let $\Omega\subset\R^2$ be a bounded simply connected Lipschitz domain with piecewise smooth boundary $\Gamma:=\partial\Omega$.
The Kirchhoff--Love plate bending model of a vertically loaded clamped plate with mid-surface $\Omega$ can be written as
\begin{align} \label{model}
   \divDiv \cC\Gradgrad u = f  \quad\text{in } \Omega,\qquad
   (u,\partial_n u) = (g,\partial_n g) \quad\text{on } \Gamma
\end{align}
with vertical deflection $u$, load $f$, and a given sufficiently smooth function $g$ on $\Omega$.
Here, $\partial_n={n}\cdot\grad$ denotes the normal derivative with exterior unit normal
vector ${n}$.
(Once defined the corresponding trace operator,
the Dirichlet data can be specified without extension $g$.)
The tensor $\bM:=-\cC\Gradgrad u$ represents the bending moments and depends linearly
on the Hessian $\Gradgrad u$ of $u$, specifically
\[
   \bM = -\cC\Gradgrad u := -D [\nu\,(\mathrm{tr}\,\Gradgrad u) \id + (1-\nu) \Gradgrad u],
\]
where
\(
   D= \frac{Ed^3}{12(1-\nu^2)}
\)
is the bending rigidity with Young's modulus $E$, Poisson ratio $\nu$,
plate thickness $d$, and $\id\in\R^{2\times 2}$ is the identity tensor.
The operator $\div$ denotes the divergence of vector functions, and $\Div$ is the divergence
operator acting on rows of tensors.

Note that the differential equation in \eqref{model}
can be written as the scaled bi-Laplace equation
\begin{align} \label{bilap}
   \div\Div\cC\Gradgrad u = D\Delta^2 u = f.
\end{align}
We set the scaling factor to $D=1$. 

\section{Spaces, traces, and jumps} \label{sec_spaces}

\subsection{Standard spaces}

We use the standard Lebesgue space $L_2(\Omega)$ with generic
$L_2(\Omega)$ duality product $\vdual{\cdot}{\cdot}_\Omega$ and
norm ${\|\cdot\|_\Omega}$ (for scalar, vector-, and tensor-valued functions)
and the Sobolev spaces $H^s(\Omega)$ for $s\ge 0$ with
canonical (Sobolev--Slobodeckij) norm $\|\cdot\|_{s,\Omega}$, cf.~\cite{McLean_00_SES},
except for $s=2$ where
\[
   \|v\|_{2,\Omega}^2 := \|v\|_\Omega^2 + \|\Gradgrad v\|_\Omega^2.
\]
For a partition $\mathcal{T}$ of $\Omega$, $H^s(\mathcal{T})$ and $\|\cdot\|_{s,\mathcal{T}}$ indicate the respective product space and norm.
We further need the extension of $\vdual{\cdot}{\cdot}_\Omega$
to the duality between $H^2(\Omega)$ and its dual space
$\widetilde H^{-2}(\Omega):=H^2(\Omega)^*$
with canonical norm $\|\cdot\|_{-2,\Omega}$.
We use the tensor space
\[
   H(\dDiv,\Omega;\SS) := \{\bQ\in L_2(\Omega;\SS);\; \divDiv\bQ\in L_2(\Omega)\}
\]
with values in $\SS:=\{\bQ\in \R^{2\times 2};\; \bQ=\bQ^\top\}$ and (squared) norm
\[
   \|\bQ\|_{\dDiv,\Omega}^2 := \|\bQ\|_\Omega^2 + \|\divDiv\bQ\|_\Omega^2.
\]
We use the analogous notation for spaces defined on $\omega\subset\R^2$
with $\Omega$ replaced by $\omega$.

For sufficiently smooth curves $\omega\subset\Gamma$,
we further require Sobolev spaces $H^s(\omega)$ for $s\ge0$ with norm denoted by
$\|\cdot\|_{s,\omega}$, defined as the canonical norm for integer $s$ and, e.g., of
Sobolev--Slobodeckij type otherwise. 
Specifically, $H^0(\omega) = L_2(\omega)$ with generic $L_2(\omega)$ duality product $\dual{\cdot}{\cdot}_\omega$.
For a partition $\cE(\omega)$ of $\omega$, $H^s(\cE(\omega))$ and $\|\cdot\|_{s,\cE(\omega)}$
indicate the respective product space and norm. We also make use of the standard
spaces $C^0$, $C^1$ on indicated sets, and the notation $\|\cdot\|_{\infty,\omega}$
for the essential supremum on $\omega$.

\subsection{Traces and induced spaces}

As shown in \cite{FuehrerHN_19_UFK}, traces of $H^2(\Omega)$ and $H(\dDiv,\Omega;\SS)$
are in duality. To be specific, we define the trace operators
\begin{align*}
   &\trGgwo:\; \left\{\begin{array}{cll}
      H^2(\Omega) & \to & H(\dDiv,\Omega;\SS)^*\\
      v & \mapsto & \dual{\trGg{v}}{\bQ}
      := \vdual{\divDiv\bQ}{v}_\Omega - \vdual{\bQ}{\Gradgrad v}_\Omega
   \end{array}\right.,\\
   &\trdDwo:\; \left\{\begin{array}{cll}
      H(\dDiv,\Omega;\SS) & \to & H^2(\Omega)^*\\
      \bQ & \mapsto & \dual{\trdD{\bQ}}{v}
      := \vdual{\divDiv\bQ}{v}_\Omega - \vdual{\bQ}{\Gradgrad v}_\Omega
   \end{array}\right..
\end{align*}
These operators induce the trace spaces
\[
   H^{3/2,1/2}(\Gamma) := \trGg{H^2(\Omega)},\quad
   H^{-3/2,-1/2}(\Gamma) := \trdD{H(\dDiv,\Omega;\SS)}
\]
with respective canonical trace norms
  $\|\cdot\|_{3/2,1/2,\Gamma}$ and $\|\cdot\|_{-3/2,-1/2,\Gamma}$
(defined as the infima over all possible extensions).
A duality between them that is consistent with the trace operators is given by
\[
   \dual{\tr{q}}{\tr{v}}
   :=
   \dual{\trGg{v}}{\bQ} = \dual{\trdD{\bQ}}{v}
   \quad \bigl(\tr{v}\in H^{3/2,1/2}(\Gamma),\ \tr{q}\in H^{-3/2,-1/2}(\Gamma)\bigr)
\]
for any $v\in H^2(\Omega)$ and $\bQ\in H(\dDiv,\Omega;\SS)$ with
$\trGg{v}=\tr{v}$ and $\trdD{\bQ}=\tr{q}$.

Let $\cE(\Omega)$ and $\cN(\Omega)$ denote the sets of
(possibly curvilinear) edges and vertices of $\Omega$, respectively.
For sufficiently smooth functions $v$ and $\bQ$, integration by parts shows that
\begin{subequations}
\begin{align}
   &\dual{\trGg{v}}{\bQ} = \dual{\trdD{\bQ}}{v}
   = \dual{{n}\cdot\Div\bQ}{v}_\Gamma - \dual{\bQ{n}}{\nabla v}_\Gamma \label{pre_dual}\\
   &\quad= \sum_{\Gamma'\in\cE(\Omega)} \dual{{n}\cdot\Div\bQ+\partial_t({t}\cdot\bQ{n})}{v}_{\Gamma'}
   + \sum_{z\in\cN(\Omega)} [{t}\cdot\bQ{n}](z) v(z) \label{dual}
   - \dual{{n}\cdot\bQ{n}}{\partial_n v}_\Gamma.
\end{align}
\end{subequations}
In that case, the traces can be identified with their respective components
$\trGg{v}=\bigl(\trD{v}, \trN{v}\bigr)$ and
$\trdD{\bQ}=\bigl(\trnn{q},\trsf{q},\trJ{q}\bigr)$ where, for $\Gamma'\in\cE(\Omega)$
and $z\in\cN(\Omega)$,
\begin{subequations}
\begin{align}
   &\trD{v}:=v|_\Gamma,\ \trN{v}:=\partial_n v|_\Gamma, \label{trGg_comp}\\
   &\trnn{q}|_{\Gamma'}:=({n}\cdot\bQ{n})|_{\Gamma'},\
   \trsf{q}|_{\Gamma'}:=\nDiv{\bQ}|_{\Gamma'},\
   \trJ{q}(z):=[{t}\cdot\bQ{n}](z) \label{trdD_comp}\\
   \text{with}\quad \label{nDiv}
   &\nDiv{\bQ}|_{\Gamma'}:=\bigl({n}\cdot\Div\bQ+\partial_t({t}\cdot\bQ{n})\bigr)|_{\Gamma'}
   \quad\text{(effective shear force)}.
\end{align}
\end{subequations}
Here, $\partial_n$ has been defined before,
and $\partial_t$ denotes the tangential derivative along $\Gamma$ in positive (i.e., counterclockwise) orientation.
The term $[{t}\cdot\bQ{n}](z)$ stands for the jump of ${t}\cdot\bQ{n}$ at $z$ in positive
orientation, that is, $[{t}\cdot\bQ{n}](z)
 := \lim_{0<s\to 0} \bigl({t}\cdot\bQ{n}(\gamma(s)) -  {t}\cdot\bQ{n}(\gamma(-s))\bigr)$
where $\gamma$ is the parametrization of $\Gamma$ with respect to the arc-length $s$
and $\gamma(0)=z$.
See~\cite{FuehrerHN_19_UFK} for details and note that we define jumps with opposite sign here.

Analogously to \cite[Corollary 2.1]{Schmidt_01_BIO}, we have the second Green identity.

\begin{lemma} \label{la_Green2}
Any $v,w\in H^2(\Omega)$ with $\Delta^2v,\Delta^2w\in L_2(\Omega)$ satisfy
\begin{align*}
   &\vdual{v}{\Delta^2 w}_\Omega - \vdual{w}{\Delta^2 v}_\Omega
   =
   \dual{\trdD{\cC\Gradgrad w}}{\trGg{v}} - \dual{\trdD{\cC\Gradgrad v}}{\trGg{w}}.
\end{align*}
\end{lemma}

\begin{proof}
The statement is due to identity \eqref{bilap} (with $D=1$), the symmetry of $\cC$,
and an application of the trace operators.
\end{proof}

Analogously to \cite[Lemmas~3.2,~3.3]{FuehrerHN_19_UFK}, the next lemma establishes that
the norms in the trace spaces are equal to the corresponding dual norms. 
For an abstract proof, we refer to \cite[Lemma~28]{Heuer_26_GMP}.

\begin{lemma}\label{lem_duality}
  Any  $\tr{u}\in H^{3/2,1/2}(\Gamma)$ and $\tr{m}\in H^{-3/2,-1/2}(\Gamma)$ satisfy
  \begin{alignat*}{3}
    \|\tr{u}\|_{3/2,1/2,\Gamma}
    &= \sup_{\bQ\in H(\dDiv,\Omega;\SS)\setminus\{0\}}
      \frac{\dual{\tr{u}}{\bQ}}{\|\bQ\|_{\dDiv,\Omega}}
    &&= \sup_{\tr{q}\in H^{-3/2,-1/2}(\Gamma) \setminus\{0\}}
      \frac{\dual{\tr{u}}{\tr{q}}}{\|\tr{q}\|_{-3/2,-1/2,\Gamma}},\\
    \|\tr{m}\|_{-3/2,-1/2,\Gamma}
    &= \sup_{v\in H^2(\Omega)\setminus\{0\}} \frac{\dual{\tr{m}}{v}}{\|v\|_{2,\Omega}}
    &&= \sup_{\tr{v}\in H^{3/2,1/2}(\Gamma) \setminus\{0\}}
      \frac{\dual{\tr{m}}{\tr{v}}}{\|\tr{v}\|_{3/2,1/2,\Gamma}}.
  \end{alignat*}
\end{lemma}

For our approximation analysis, we need an intrinsic norm in $H^{3/2,1/2}(\Gamma)$,
\begin{equation} \label{tnorm}
   \tnorm{\tr{v}}_{3/2,1/2,\Gamma}^2
   :=\|\trD{v}\|_{1,\Gamma}^2
   + \|(\partial_t \trD{v},\trN{v})^\top\cdot{t}\|_{1/2,\Gamma}^2
   + \|(\partial_t \trD{v},\trN{v})^\top\cdot{n}\|_{1/2,\Gamma}^2
\end{equation}
for $\tr{v}=(\trD{v},\trN{v})\in H^{3/2,1/2}(\Gamma)$, cf.~\cite{Jakovlev_61_BPF,SchmidtK_99_BIE}.

\begin{lemma} \label{lem_isomorphic}
The trace space $(H^{3/2,1/2}(\Gamma), \|\cdot\|_{3/2,1/2,\Gamma})$ is isomorphic to
\(
  \{\tr{v}=(\trD{v},v_{n});\; \tnorm{\tr{v}}_{3/2,1/2,\Gamma}<\infty\}.
\)
In particular, any $\tr{v}=(\trD{v},\trN{v})\in H^{3/2,1/2}(\Gamma)$ satisfies
\begin{align} \label{trestH2}
    \sum_{{\Gamma'}\in\cE(\Omega)}
    \|\partial_t\trD{v}\|_{1/2,{\Gamma'}}^2 + \|\trN{v}\|_{1/2,{\Gamma'}}^2
    \lesssim \|\tr{v}\|_{3/2,1/2,\Gamma}^2.
\end{align}
\end{lemma}

\begin{proof}
Let us adopt, only for this proof, the notation $V(\Gamma)$ from \cite{SchmidtK_99_BIE}
for the claimed isomorphic space.
By definition and Lemma~\ref{lem_duality}, $\trGgwo:\;H^2(\Omega)\to H^{3/2,1/2}(\Gamma)$ is
bounded and surjective with continuous right inverse.
Furthermore, by \cite[Lemma~2.2]{SchmidtK_99_BIE} (see also \cite{Jakovlev_61_BPF}),
$I:\;H^2(\Omega)\to V(\Gamma)$ defined by $I(v):=(v|_\Gamma,\partial_n v|_\Gamma)$
is also bounded and surjective with continuous right inverse.
Since both operators coincide on a dense subset of smooth functions, cf.~\eqref{trGg_comp},
the isomorphic relation of both trace spaces follows.
Estimate \eqref{trestH2} is a direct consequence of the bound
\begin{align*}
    \sum_{{\Gamma'}\in \cE(\Omega)} \|v|_\Gamma\|_{3/2,{\Gamma'}}^2 + 
    \|\partial_n v\|_{1/2,{\Gamma'}}^2 \lesssim \|v\|_{2,\Omega}^2 \quad \forall v\in H^2(\Omega),
\end{align*}
see~\cite[Lemma~2.1]{SchmidtK_99_BIE}, and the definition of the trace space and operator.
\end{proof}

\begin{remark} \label{rem_Ggrad}
Some further comments on the trace space $H^{3/2,1/2}(\Gamma)$ are in order.

{\rm (i)} For $v\in H^2(\Omega)$ and $(\trD{v},\trN{v})=\trGg{v}$, the expressions
in $\tnorm{(\trD{v},\trN{v})}_{3/2,1/2,\Gamma}$ are
$(\partial_t \trD{v},\trN{v})^\top\cdot{t}=\partial_{x_1} v$ and
$(\partial_t \trD{v},\trN{v})^\top\cdot{n}=\partial_{x_2} v$, the derivatives of
$v$ in coordinate directions $x_1$, $x_2$, i.e.,
\[
   \tnorm{(\trD{v},\trN{v})}_{3/2,1/2,\Gamma}^2
   = \|v\|_\Gamma^2 + \|{t}\cdot\grad v\|_\Gamma^2 + \|\grad v\|_{1/2,\Gamma}^2.
\]

{\rm(ii)} For smooth curves $\Gamma$, representation \eqref{trGg_comp}
and canonical trace arguments imply that $H^{3/2,1/2}(\Gamma)$
is isomorphic to the product space $H^{3/2}(\Gamma)\times H^{1/2}(\Gamma)$.
\end{remark}

\subsection{Jumps}

Let $\Omega^e:=\R^2\setminus\overline\Omega$ be the exterior domain.
Given any fixed function $\varphi\in C^\infty_0(\R^2)$
with $\varphi=1$ in a neighborhood of $\Gamma$, we define the jump and trace-sum functionals
\begin{align*}
   &\jGg{\cdot}:\; \left\{\begin{array}{cll}
      H^2(\R^2\setminus\Gamma) & \to & H(\dDiv,\R^2;\SS)^*\\
      v & \mapsto & \dual{\jGg{v}}{\bQ}
      := -\vdual{\divDiv\bQ}{v\varphi}_{\R^2} + \vdual{\bQ}{\Gradgrad(v\varphi)}_{\Omega\cup\Omega^e}
   \end{array}\right.,\\[1em]
   &\jdD{\cdot}:\; \left\{\begin{array}{cll}
      H(\dDiv,\R^2\setminus\Gamma;\SS) & \to & H^2(\R^2)^*\\
      \bQ & \mapsto & \dual{\jdD{\bQ}}{v}
      := -\vdual{\divDiv\bQ}{v\varphi}_{\Omega\cup\Omega^e} + \vdual{\bQ}{\Gradgrad(v\varphi)}_{\R^2}
   \end{array}\right.,\\[1em]
   &\avGg{\cdot}:\; \left\{\begin{array}{cll}
      H^2(\R^2\setminus\Gamma) & \to & H(\dDiv,\R^2;\SS)^*\\
      v & \mapsto & \dual{\avGg{v}}{\bQ}
      := \vdual{\divDiv\bQ}{v\varphi}_\Omega
                       - \vdual{\bQ}{\Gradgrad(v\varphi)}_\Omega\\
      &&\hspace{6.3em}
      \quad- \vdual{\divDiv\bQ}{v\varphi}_{\Omega^e} + \vdual{\bQ}{\Gradgrad(v\varphi)}_{\Omega^e}
   \end{array}\right..
\end{align*}
Note the sign change in the jumps in comparison with the trace operators,
implying that the jumps correspond to the traces ``taken from outside'' minus the interior traces.
Here, the $H^2$ trace ``taken from outside'' means the negative exterior trace
\[
   \trGge{v}:=\frac 12\Bigl(\jGg{v}+\avGg{v}\Bigr) \quad \bigl(v\in H^2(\R^2\setminus\Gamma)\bigr),
\]
to account for the change of the normal direction ${n}$ pointing into $\Omega^e$.
The trace $\trdDe{\bQ}$ can be defined analogously for $\bQ\in H(\dDiv,\R^2\setminus\Gamma;\SS)$.
As before, there are dualities
$\dual{\jGg{v}}{\tr{q}}$ and $\dual{\jdD{\bQ}}{\tr{v}}$
between the corresponding
jumps of $v\in H^2(\R^2\setminus\Gamma)$ and $\bQ\in H(\dDiv,\R^2\setminus\Gamma;\SS)$, and
traces $\tr{q}\in H^{-3/2,-1/2}(\Gamma)$ and $\tr{v}\in H^{3/2,1/2}(\Gamma)$.

\section{Layer potentials and boundary integral operators} \label{sec_BIO}

A fundamental solution of $\Delta^2 u=\divDiv\Gradgrad u=0$ is given by
\[
   G(x) := \frac 1{8\pi} |x|^2\log |x|\qquad (x\in\R^2\setminus\{0\})
\]
with extension by continuity $G(0):=0$.
Throughout, $|\cdot|$ stands for the Euclidean norm in $\R^2$.
An application of the second Green identity from Lemma~\ref{la_Green2}
(which holds analogously in $\Omega^e$) gives the following representation formula,
cf.~\cite[Lemma~3.1]{Schmidt_01_BIO}.

\begin{lemma}
Any function $u\in L_2(\R^2)$ with compact support that satisfies
$u|_\Omega\in H^2(\Omega)$, $u|_{\Omega^e}\in H^2_\mathrm{loc}(\Omega^e)$, and
$\Delta^2 u|_{\R^2\setminus\Gamma}=f\in L_2(\R^2)$, has
at $x\in\R^2\setminus\Gamma$ the representation
\[
   u(x) = \dual{\jdD{\cC\Gradgrad u}}{\trGg{G(x-\cdot)}}
        - \dual{\jGg{u}}{\trdD{\cC\Gradgrad G(x-\cdot)}}
    + \dual{f}{G(x-\cdot)}_\Omega.
\]
\end{lemma}

This representation\footnote{and the sign convention in the representation formula of the Laplacian}
motivates the definition of the single-layer and double-layer potentials
\begin{alignat*}{2}
   &\Vpot\tr{q}(x) := \dual{\tr{q}}{\trGg{G(x-\cdot)}}
   &&\bigl(\tr{q}\in H^{-3/2,-1/2}(\Gamma),\ x\in\R^2\setminus\Gamma\bigr),\\
   &\Kpot\tr{v}(x) := -\dual{\trdD{\cC\Gradgrad G(x-\cdot)}}{\tr{v}}
   \quad
   &&\bigl(\tr{v}\in H^{3/2,1/2}(\Gamma),\ x\in\R^2\setminus\Gamma\bigr),
\end{alignat*}
and we introduce the Newton potential for functionals
with support in $\overline\Omega$,
\[
   \Npot f(x) := \dual{f}{G(x-\cdot)}_\Omega\qquad
   (f\in \widetilde H^{-2}(\Omega),\ x\in\R^2\setminus\Gamma),
\]
even though for simplicity we only consider $L_2(\Omega)$-loads.

\begin{lemma}[{\cite[Lemma~3.2,~Corollary~3.1, (3.1), (3.6)]{Schmidt_01_BIO}}] \label{la_rep}
The operators $\Vpot:\; H^{-3/2,-1/2}(\Gamma) \to H^2_\mathrm{loc}(\R^2)$,
$\Kpot:\; H^{3/2,1/2}(\Gamma) \to H^2_\mathrm{loc}(\Omega\cup\Omega^e)$,
and $\Npot:\;\widetilde H^{-2}(\Omega)\to H^2_\mathrm{loc}(\R^2)$ are bounded.
Furthermore, $u\in H^2(\Omega)$ with $\Delta^2 u=f\in L_2(\Omega)$ satisfies
\[
   \Vpot\trdD{-\cC\Gradgrad u} - \Kpot\trGg{u}
   + \Npot f
   = \begin{cases} u &\text{in}\ \Omega,\\ 0 & \text{in}\ \Omega^e.\end{cases}
\]
\end{lemma}

\begin{lemma}[{\cite[Lemma~3.5]{Schmidt_01_BIO}}] \label{la_jump}
The operators $\Vpot$ and $\Kpot$ satisfy the jump relations
\begin{align*}
   \jGg{\Vpot\tr{q}} = 0\quad \forall\tr{q}\in H^{-3/2,-1/2}(\Gamma),\qquad
   \jGg{\Kpot\tr{v}} = \tr{v}\quad \forall\tr{v}\in H^{3/2,1/2}(\Gamma).
\end{align*}
\end{lemma}

We define the single-layer and double-layer integral operators
\begin{alignat*}{2}
   &\V:\; H^{-3/2,-1/2}(\Gamma)\to H^{3/2,1/2}(\Gamma),\quad
   &&\V\tr{q} := (\Vo\tr{q},\Vt\tr{q}) := \trGg{\Vpot\tr{q}},\\
   &\K:\; H^{3/2,1/2}(\Gamma)\to H^{3/2,1/2}(\Gamma),\quad
   &&\K\tr{v} := (\Ko\tr{v},\Kt\tr{v}) := \frac 12 \avGg{\Kpot\tr{v}},
\end{alignat*}
cf.~the notation from \eqref{trGg_comp}. Note that by the jump relation of $\Kpot$,
the definition of $\K$ is equivalent to the identity
\begin{equation} \label{trace_DL}
   \K\tr{v} = \trGg{\Kpot\tr{v}}+\frac 12\tr{v}
\end{equation}
for $\tr{v}\in H^{3/2,1/2}(\Gamma)$.
We also formally introduce the trace of the Newton potential
\[
    \NIO:\; L_2(\Omega) \to H^{3/2,1/2}(\Gamma),\quad
    \NIO f:=(\No f,\Nt f):=\trGg{\Npot f}.
\]
Then, as a consequence of Lemmas~\ref{la_rep},~\ref{la_jump} and relation \eqref{trace_DL},
we conclude that the traces $\tr{u}:=\trGg{u}$ and $\tr{m}:=\trdD{\bM}$
of the solution $u$ to \eqref{model} and $\bM:=-\cC\Gradgrad u$
satisfy the boundary integral equation with Newton potential
\begin{align} \label{BIE}
   \V\tr{m} = \K\tr{u} + \frac 12\tr{u} - \NIO f.
\end{align}
The next two theorems provide integral representations of both operators,
first for the single-layer and second for the double-layer operator.
For their proofs, we use standard arguments from the theory of boundary integral operators,
see, e.g., \cite{McLean_00_SES,Steinbach_08_NAM}.
In case of smooth boundaries $\Gamma$, such representations are already stated in~\cite[Section~10.4.4]{HsiaoW_21_BIE}.

\begin{theorem}\label{thm_BIO_V}
Let $\bQ\in H(\dDiv,\Omega;\SS)$ with trace $\tr{q}:=\trdD{\bQ}$ allowing for a representation $\tr{q} = (\trnn{q},\trsf{q},\trJ{q}) \in L_\infty(\Gamma) \times L_\infty(\Gamma) \times \R^{\#\cN(\Omega)}$,
cf.~\eqref{trdD_comp}.
At $x\in\Gamma$, the single-layer operator $\V\tr{q}=(\Vo\tr{q},\Vt\tr{q})$
has the representation
\begin{align}
\begin{split} \label{V1_rep}
   \Vo\tr{q}(x)
   &= \int_\Gamma \trsf{q}(y) G(x-y)\ds_y
   - \int_\Gamma \trnn{q}(y) \partial_{n_y} G(x-y)\ds_y
   + \sum_{z\in\cN(\Omega)} \trJ{q}(z) G(x-z)
   \\&= \frac 1{8\pi} \int_\Gamma \trsf{q}(y) |x-y|^2\log|x-y|\ds_y
    + \frac 1{8\pi} \int_\Gamma \trnn{q}(y) (2\log|x-y|+1){n}_y\cdot({x}-{y})\ds_y
   \\&\quad
   + \frac 1{8\pi} \sum_{z\in\cN(\Omega)} \trJ{q}(z) |x-z|^2\log|x-z|,
\end{split}
\\
\begin{split}
   \Vt\tr{q}(x)
   &= \int_\Gamma \trsf{q}(y) \partial_{n_x} G(x-y)\ds_y
   - \int_\Gamma \trnn{q}(y) \partial_{n_x}\partial_{n_y} G(x-y)\ds_y
   + \sum_{z\in\cN(\Omega)} \trJ{q}(z) \partial_{n_x}G(x-z)
   \\&=
   \frac 1{8\pi}\int_\Gamma \trsf{q}(y) (2\log|x-y|+1){n}_x\cdot({x}-{y})\ds_y
   \\&\quad
     + \frac 1{8\pi} \int_\Gamma \trnn{q}(y)
        \Bigl((2\log|x-y|+1){n}_x\cdot{n}_y
            + 2\frac{{n}_x\cdot({x}-{y})\,{n}_y\cdot({x}-{y})}{|x-y|^2}
        \Bigr) \ds_y
   \\&\quad
   + \frac 1{8\pi} \sum_{z\in\cN(\Omega)} \trJ{q}(z) (2\log|x-z|+1){n}_x\cdot({x}-z)
   \qquad (x\not\in\cN(\Omega)).
\end{split}
   \label{V2_rep}
\end{align}
\end{theorem}

\begin{proof}
Let $x\in\Gamma$ (fixed) and $\wilde x\in\Omega$. In this and the following proof,
we abbreviate $\HG:=\cC\Gradgrad G$ and make use of the calculations
\begin{equation} \label{calc}
\begin{split}
   \grad G(x) &= \frac 1{8\pi}(2\log|x|+1){x},\quad
   \partial_n G(x) = \frac 1{8\pi}(2\log|x|+1){n}\cdot{x},\\
   \Delta G(x) &= \frac 1{2\pi}(\log|x|+1),\quad
   \Gradgrad G = \frac 1{8\pi}(2\log|x|+1)\id + \frac 1{4\pi} \frac{xx^\top}{|x|^2},\\
   \partial_{n_x}\partial_{n_y} G(x-y)
   &= -{n}_x\cdot\Gradgrad G(x-y){n}_y\\
   &= -\frac 1{8\pi} (2\log|x-y|+1){n}_x\cdot{n}_y
   - \frac 1{4\pi} \frac{{n}_x\cdot({x}-{y})\, {n}_y\cdot({x}-{y})}{|x-y|^2},\\
   \HG
   &= \nu\mathrm{tr}(\Gradgrad G)\id+(1-\nu)\Gradgrad G
   = \nu\Delta G\id + (1-\nu) \Gradgrad G,\\
   {n}\cdot\HG{n}(x)
   &= \frac \nu{2\pi} (\log|x|+1)
    + \frac {1-\nu}{8\pi} (2\log|x|+1)
    + \frac {1-\nu}{4\pi} \frac {({n}\cdot {x})^2}{|x|^2}\\
   &= \frac {1+3\nu}{8\pi} + \frac {1+\nu}{4\pi} \log|x|
    + \frac {1-\nu}{4\pi} \frac {({n}\cdot {x})^2}{|x|^2},\\
   {t}\cdot\HG{n}(x)
   &= \frac {1-\nu}{4\pi} \frac {{t}\cdot {x}\, {n}\cdot x}{|x|^2},\quad
   {n}\cdot\Div\HG(x)
   = \partial_n\Delta G(x) = \frac 1{2\pi}\partial_n\log|x|.
\end{split}
\end{equation}
According to the definitions of $\V$ and $\Vpot$, we have to show that
\begin{align*}
   \trGg{\wilde x\mapsto \dual{\tr{q}}{\trGg{G(\wilde x-\cdot)}}}
   &= \bigl( \dual{\tr{q}}{\trGg{G(\wilde x-\cdot)}},\;
            \partial_{n_x}\dual{\tr{q}}{\trGg{G(\wilde x-\cdot)}} \bigr)
\end{align*}
tends to the claimed representations of $\bigl(\Vo\tr{q}(x),\Vt\tr{q}(x)\bigr)$
as $\wilde x\to x$ (the second component only for $x\not\in\cN(\Omega)$).
Here and in what follows, $\partial_{n_x}$ of a function depending on $\wilde x$
denotes the directional derivative of the function with respect to $\wilde x$
in direction of the normal vector ${n}$ at $x\in\Gamma\setminus\cN(\Omega)$.

We start with \eqref{V1_rep}.  According to \eqref{dual}, we have
\begin{multline*}
   \dual{\tr{q}}{\trGg{G(\wilde x-\cdot)}}\\
   = \int_\Gamma \trsf{q}(y)G(\wilde x-y) \ds_y
   + \sum_{z\in\cN(\Omega)} \trJ{q}(z) G(\wilde x-z)
   - \int_\Gamma \trnn{q}(y) \partial_{n_y} G(\wilde x-y) \ds_y.
\end{multline*}
The continuity of $G$, $\nabla G$, the boundedness of the data,
and the expression for $\partial_n G(x)$ from \eqref{calc}
lead to the claimed representations of limit \eqref{V1_rep}.
In case of relation \eqref{V2_rep}, the same argument applies by using the integrability
of the kernel $\partial_{n_x}\partial_{n_y}G(x-y)$.
\end{proof}

\begin{theorem}\label{thm_BIO_K}
Let $v\in H^2(\Omega)$ with trace $\tr{v}:=\trGg{v}=(\trD{v},\trN{v})$, cf.~\eqref{trGg_comp}.
Assume that, on every edge $\Gamma'\in\cE(\Omega)$,
 $\trD{v}$ is continuously differentiable and $\trN{v}$ is continuous.
At $x\in\Gamma$, the double-layer operator $\K\tr{v}=(\Ko\tr{v},\Kt\tr{v})$ has the representation
\begin{align}
\begin{split} \label{K1_rep}
   \Ko\tr{v}(x)
   &=
   - \int_\Gamma \trD{v}(y) {n}_y\cdot\Div_y\bigl(\cC\Gradgrad G(x-y)\bigr) \ds_y
   + \int_\Gamma {t}_y\cdot\bigl(\cC\Gradgrad G(x-y)\bigr){n}_y
                 \partial_{t_y}\trD{v}(y) \ds_y
   \\ &\quad
   + \int_\Gamma \trN{v}(y) {n}_y\cdot\bigl(\cC\Gradgrad G(x-y)\bigr){n}_y \ds_y
   + \bigl(\sigma(x)-\frac 12\bigr) \trD{v}(x)
   \\
   &=
   -\frac 1{2\pi}       \int_\Gamma \trD{v}(y) \partial_{n_y}\log|x-y| \ds_y
   +\frac {1-\nu}{4\pi} \int_\Gamma \frac{{t}_y\cdot({x}-{y})\, {n}_y\cdot({x}-{y})}{|x-y|^2}
                                    \partial_{t_y}\trD{v}(y) \ds_y
   \\ &\quad
   +  \int_\Gamma \trN{v}(y)
           \Bigl(
                  \frac {1+3\nu}{8\pi} + \frac {1+\nu}{4\pi} \log|x-y|
                + \frac {1-\nu}{4\pi} \frac {\bigl({n}_y\cdot ({x}-{y})\bigr)^2}{|x-y|^2}
           \Bigr) \ds_y
   + \bigl(\sigma(x)-\frac 12\bigr) \trD{v}(x),
\end{split}
\\
\begin{split} \label{K2_rep}
   \Kt\tr{v}(x)
   &=
   - \int_\Gamma \trD{v}(y) \partial_{n_x}{n}_y\cdot\Div_y\bigl(\cC\Gradgrad G(x-y)\bigr) \ds_y
   + \int_\Gamma \partial_{n_x}{t}_y\cdot\bigl(\cC\Gradgrad G(x-y)\bigr){n}_y
                  \partial_{t_y} \trD{v}(y) \ds_y
   \\ &\quad
   + \int_\Gamma \trN{v}(y) \partial_{n_x}{n}_y\cdot\bigl(\cC\Gradgrad G(x-y)\bigr){n}_y \ds_y
   \\ &=
   - \frac 1{2\pi} \int_{\Gamma}
                   \bigl(\trD{v}(y)-\trD{v}(x)\bigr) \partial_{n_x}\partial_{n_y}\log|x-y| \ds_y
   \\ &\quad
   + \frac {1-\nu}{4\pi} \int_{\Gamma}
                         \partial_{n_x} \frac{{t}_y\cdot({x}-{y}) {n}_y\cdot({x}-{y})}{|x-y|^2}
                         \partial_{t_y} \trD{v}(y) \ds_y
   \\ &\quad
   + \int_{\Gamma} \trN{v}(y)
          \Bigl(\frac{1+\nu}{4\pi} \partial_{n_x}\log|x-y|
                + \frac {1-\nu}{4\pi} \partial_{n_x}
                  \frac{({n}_y\cdot ({x}-{y}))^2}{|x-y|^2}
          \Bigr) \ds_y
   \qquad (x\not\in\cN(\Omega)),
\end{split}
\end{align}
where the strongly singular integrals are taken as their Cauchy principal values.
Here,
\begin{align*}
   \sigma(x) := \frac 1{2\pi} \lim_{\epsilon\to 0}
                \frac 1\epsilon \int_{y\in\Omega:\;|y-x|=\epsilon} \ds_y
\end{align*}
is the interior angle at $x\in\Gamma$ divided by $2\pi$.
\end{theorem}

\begin{proof}
As in the proof of Theorem~\ref{thm_BIO_V}, $x\in\Gamma$ is fixed and $\wilde x\in\Omega$.
Recall the abbreviation $\HG:=\cC\Gradgrad G$ and calculations \eqref{calc}.
According to the definitions of $\Kpot$ and $\K$, and relation \eqref{trace_DL},
we have to show that
\begin{align} \label{pf_K_rep}
    \frac 12\tr{v}(x) 
    - \Bigl(\dual{\trdD{\HG(\wilde x-\cdot)}}{\tr{v}},
            \partial_{n_x}\dual{\trdD{\HG(\wilde x-\cdot)}}{\tr{v}}\Bigr)
\end{align}
tends to the claimed representations of $\bigl(\Ko\tr{v}(x),\Kt\tr{v}(x)\bigr)$ as $\wilde x\to x$
($x\not\in\cN(\Omega)$ in case of the second component).

\medskip\noindent {\bf First component of $\K$.}
We use relation \eqref{pre_dual} to rewrite
\begin{align*}
   &-\dual{\trdD{\HG(\wilde x-\cdot}}{\tr{v}}
   = -\dual{{n}\cdot\Div \HG(\wilde x-\cdot)}{v}_\Gamma
   + \dual{\HG(\wilde x-\cdot){n}}{\nabla v}_\Gamma.
\end{align*}
The last term has an integrable kernel so that, as seen previously for $\V$,
\[
   \dual{\HG(\wilde x-\cdot){n}}{\nabla v}_\Gamma
   \to 
   \dual{\HG(x-\cdot){n}}{\nabla v}_\Gamma \quad (\wilde x\to x).
\]
For the other term we use the calculations \eqref{calc} to see that
\begin{align*}
   -\dual{{n}\cdot\Div \HG(\wilde x-\cdot)}{v}_\Gamma
   = -\frac 1{2\pi} \int_\Gamma v(y) \partial_{n_y} \log|\wilde x-y| \ds_y
\end{align*}
is identical to the double-layer potential of the Laplacian, applied to the restriction
of $v$ to $\Gamma$.  We conclude that
\[
   -\dual{{n}\cdot\Div \HG(\wilde x-\cdot)}{v}_\Gamma
   \to
   -\frac 1{2\pi} \int_\Gamma v(y) \partial_{n_y} \log|x-y|\ds_y
   + (\sigma(x)-1) v(x)\quad (\wilde x\to x),
\]
where the singular integral is interpreted as its Cauchy principal value,
cf.~\cite[Lemma~6.11]{Steinbach_08_NAM}.
A combination of the findings with representation \eqref{trace_DL},
the splitting of $\grad v$ into its normal and tangential components,
and the relations for the tangential-normal and normal-normal traces of $\HG$
in \eqref{calc} show that
\begin{align*}
   \Ko\tr{v}(x)
   &= -\frac 1{2\pi} \int_\Gamma v(y) \partial_{n_y} \log|x-y|\ds_y
   + \bigl(\sigma(x)-\frac 12\bigr) v(x)
   + \int_\Gamma \HG(x-y){n}_y \cdot\nabla_y v(y) \ds_y\\
   \\&= 
     -\frac 1{2\pi}
      \int_\Gamma \trD{v}(y) \partial_{n_y}\log|x-y| \ds_y
     + \frac {1-\nu}{4\pi}
      \int_\Gamma \frac{{t}_y\cdot({x}-{y})\, {n}_y\cdot({x}-{y})}{|x-y|^2}
                  \partial_{t_y} \trD{v}(y) \ds_y
   \\ &\quad
   + \int_\Gamma \Bigl(
                  \frac {1+3\nu}{8\pi} + \frac {1+\nu}{4\pi} \log|x-y|
                + \frac {1-\nu}{4\pi} \frac {\bigl({n}_y\cdot ({x}-{y})\bigr)^2}{|x-y|^2}
                 \Bigr) \trN{v}(y) \ds_y
   + \bigl(\sigma(x)-\frac 12\bigr) \trD{v}(x).
\end{align*}
This is the limit of the first component in \eqref{pf_K_rep} and proves \eqref{K1_rep}.

\medskip\noindent {\bf Second component of $\K$.}
It remains to establish the limit of the second component in \eqref{pf_K_rep}.
We fix $x\in\Gamma$ with $x\not\in\cN(\Omega)$. For
$\Geps{x}:=\{y\in\Gamma;\; |x-y|<\epsilon\}$ ($\epsilon>0$), we define
\begin{align*}
   \Kteps\tr{v}(x)
   &:=
   \frac {1+\nu}{4\pi}
   \int_{\Gamma\setminus\Geps{x}} \trN{v}(y) \partial_{n_x}\log|x-y| \ds_y
   + \frac {1-\nu}{4\pi}
     \int_{\Gamma\setminus\Geps{x}}
          \trN{v}(y) \partial_{n_x} \frac{({n}_y\cdot ({x}-{y}))^2}{|x-y|^2} \ds_y
   \\ &\quad
   - \frac 1{2\pi}
     \int_{\Gamma\setminus\Geps{x}} \bigl(\trD{v}(y)-\trD{v}(x)\bigr)
                                    \partial_{n_x} \partial_{n_y}\log|x-y| \ds_y
   \\ &\quad
   + \frac {1-\nu}{4\pi}
     \int_{\Gamma\setminus\Geps{x}}
     \partial_{n_x}\frac{{t}_y\cdot({x}-{y})\, {n}_y\cdot({x}-{y})}{|x-y|^2} \partial_{t_y}\trD{v}(y) \ds_y
\end{align*}
and note that $\Kteps\tr{v}(x)$ tends to the claimed representation of 
$\Kt\tr{v}(x)$ as $\epsilon\to 0$.
For the limit of the second component in \eqref{pf_K_rep} we therefore have to show that
\begin{align} \label{pf_Kc}
   \partial_{n_x} \dual{\trdD{\HG(\wilde x-\cdot)}}{\tr{v}}
   + \Kteps\tr{v}(x)
   \to \frac 12 \trN{v}(x) \quad (\wilde x\to x, \epsilon\to 0).
\end{align}
The representation $\Kpot\trGg{1}=-1$ in $\Omega$ by Lemma~\ref{la_rep} implies
\[
   \grad \int_\Gamma \partial_{n_y}\log|\wilde x-y|\ds_y = 0,
\]
cf.~\cite[(6.17)]{Steinbach_08_NAM}).
Therefore, as before, making use of relations \eqref{pre_dual} and \eqref{calc},
and splitting the gradient $\grad v$ into its tangential and normal components,
\begin{align} \label{pf_Ka}
   &\partial_{n_x}\dual{\trdD{\HG(\wilde x-\cdot)}}{\tr{v}}
   =
   \partial_{n_x}
   \int_\Gamma
      {n}_y\cdot \Div{\HG(\wilde x-y)} v(y)
      -\HG(\wilde x-y){n}_y \cdot\grad_y v(y)
   \ds_y
   \nonumber\\
   &=
     \frac 1{2\pi} \int_\Gamma \bigl(\trD{v}(y)-\trD{v}(x)\bigr)
                   \partial_{n_x} \partial_{n_y}\log|\wilde x-y| \ds_y
   - \frac {1-\nu}{4\pi}
     \int_\Gamma  \partial_{n_x}
        \frac{{t}_y\cdot(\wilde {x}-{y})\, {n}_y\cdot(\wilde {x}-{y})}{|\wilde x-y|^2}
        \partial_{t_y} \trD{v}(y)\ds_y \nonumber\\
   &\quad
   - \frac {1+\nu}{4\pi}
     \int_\Gamma \trN{v}(y) \partial_{n_x} \log|\wilde x-y| \ds_y
   - \frac {1-\nu}{4\pi}
     \int_\Gamma
     \trN{v}(y) \partial_{n_x}
     \frac{({n}_y\cdot (\wilde {x}-{y}))^2}{|\wilde x-y|^2} \ds_y.
\end{align}
We conclude that
\begin{align}
   &\partial_{n_x} \dual{\trdD{\HG(\wilde x-\cdot)}}{\tr{v}}
   + \partial_{n_x} \Keps\tr{v}(x)
   \nonumber \\
   \begin{split} \label{pf_Kd}
   &= \frac 1{2\pi} \int_{\Gamma\setminus\Geps{x}}
         \bigl(\trD{v}(y)-\trD{v}(x)\bigr)
         \partial_{n_x}\partial_{n_y}\bigl(\log|\wilde x-y|-\log|x-y|\bigr) \ds_y
   \\ &\quad
   - \frac {1-\nu}{4\pi}
     \int_{\Gamma\setminus\Geps{x}} \partial_{n_x}
        \Bigl(\frac{{t}_y\cdot(\wilde {x}-{y})\, {n}_y\cdot(\wilde {x}-{y})}{|\wilde x-y|^2}
             -\frac{{t}_y\cdot({x}-{y})\, {n}_y\cdot({x}-{y})}{|x-y|^2}\Bigr)
        \partial_{t_y} \trD{v}(y)
     \ds_y
   \\ &\quad
   - \frac {1+\nu}{4\pi} \int_{\Gamma\setminus\Geps{x}} \trN{v}(y)
        \partial_{n_x}\bigl(\log|\wilde x-y| - \log|x-y|\bigr) \ds_y
   \\ &\quad
   - \frac {1-\nu}{4\pi} \int_{\Gamma\setminus\Geps{x}} \trN{v}(y) \partial_{n_x}
        \Bigl( \frac{({n}_y\cdot (\wilde {x}-{y}))^2}{|\wilde x-y|^2}
              -\frac{({n}_y\cdot ({x}-{y}))^2}{|x-y|^2}
        \Bigr) \ds_y
   \\ &\quad
   + \frac 1{2\pi} \int_\Geps{x} \bigl(\trD{v}(y)-\trD{v}(x)\bigr)
                   \partial_{n_x}\partial_{n_y}\log|\wilde x-y| \ds_y
   \\ &\quad
   - \frac {1-\nu}{4\pi}
     \int_\Geps{x}
        \partial_{n_x} \frac{{t}_y\cdot(\wilde {x}-{y})\, {n}_y\cdot(\wilde {x}-{y})}
                            {|\wilde x-y|^2}
        \partial_{t_y} \trD{v}(y)\ds_y
   \\ &\quad
   - \frac {1+\nu}{4\pi} \int_\Geps{x} \trN{v}(y) \partial_{n_x} \log|\wilde x-y| \ds_y
   - \frac {1-\nu}{4\pi}
     \int_\Geps{x} \trN{v}(y)
                   \partial_{n_x} \frac{({n}_y\cdot (\wilde {x}-{y}))^2}{|\wilde x-y|^2}
                   \ds_y.
   \end{split}
\end{align}
By \eqref{pf_Kc}, we have to show that these terms tend to $\trN{v}(x)/2$ when $\wilde x\to x$
and $\epsilon\to 0$.

The first four integrals tend to zero for $\wilde x\to x$,
and it remains to determine the limits of the last four integrals,
denoted as $I_j(\wilde x,\epsilon)$ ($j=1,\ldots,4$),
when $\wilde x\to x$ and $\epsilon\to 0$.

\noindent{\bf Limit of $I_1(\wilde x,\epsilon)$.}
We find that
\begin{align*}
   I_1(\wilde x,\epsilon)
   &=
   \frac 1{2\pi}
   \int_\Geps{x} \bigl(\trD{v}(y)-\trD{v}(x)\bigr) \partial_{n_x}\partial_{n_y}\log|\wilde x-y|\ds_y\\
   &\to
   \frac 1{2\pi}
   \int_\Geps{x}
   \bigl(\trD{v}(y)-\trD{v}(x)\bigr) \partial_{n_x}\partial_{n_y}\log|x-y|\ds_y
   \quad (\wilde x\to x)
\end{align*}
with the integral interpreted as its Cauchy principal value
(without the correction $\trD{v}(x)$, it is a hypersingular integral that
can be defined by its Hadamard finite part), see~\cite[Section~6.5]{Steinbach_08_NAM}.
Therefore,
$\lim_{\wilde x\to x} I_1(\wilde x,\epsilon) \to 0$ for $\epsilon\to 0$.

\noindent{\bf Limit of $I_2(\wilde x,\epsilon)$.}
We add and subtract $\partial_{t_x}\trD{v}(x)$ to write
\begin{align*}
   I_2(\wilde x,\epsilon)
   &=
   - \frac {1-\nu}{4\pi}
     \int_\Geps{x}
        \partial_{n_x} \frac{{t}_y\cdot(\wilde {x}-{y})\, {n}_y\cdot(\wilde {x}-{y})}
                            {|\wilde x-y|^2}
        \bigl(\partial_{t_y} \trD{v}(y) - \partial_{t_x}\trD{v}(x)\bigr)\ds_y\\
   &\quad
   - \frac {1-\nu}{4\pi} \partial_t\trD{v}(x)
     \int_\Geps{x}
        \partial_{n_x} \frac{{t}_y\cdot(\wilde {x}-{y})\, {n}_y\cdot(\wilde {x}-{y})}
                            {|\wilde x-y|^2} \ds_y.
\end{align*}
The first integral on the right-hand side has a bounded integrand
for $\wilde x\to x$ and, thus, tends to zero when $\wilde x\to x$ and $\epsilon\to 0$.
Direct calculations show that the second integral tends to zero when $\wilde x\to x$
so that a combination yields
$I_2(\wilde x,\epsilon)\to 0$ when $\wilde x\to x$ and $\epsilon\to 0$.

\noindent{\bf Limit of $I_3(\wilde x,\epsilon)+I_4(\wilde x,\epsilon)$.}
We determine the limits of $I_3$ and $I_4$ as their sum. To this end,
we add and subtract $\trN{v}(x)$ in the arguments to rewrite
\begin{align} \label{pf_K_I34}
\begin{split}
   I_3(\wilde x,\epsilon) + & I_4(\wilde x,\epsilon)
   =
   - \frac {1+\nu}{4\pi} \int_\Geps{x} \partial_{n_x} \log|\wilde x-y| (\trN{v}(y)-\trN{v}(x)) \ds_y
   \\ &
   - \frac {1-\nu}{4\pi}
     \int_\Geps{x} \partial_{n_x} \frac{({n}_y\cdot (\wilde {x}-{y}))^2}{|\wilde x-y|^2}
                   (\trN{v}(y)-\trN{v}(x)) \ds_y
   \\ &
   - \frac {1+\nu}{4\pi} \trN{v}(x) \int_\Geps{x} \partial_{n_x} \log|\wilde x-y| \ds_y
   - \frac {1-\nu}{4\pi} \trN{v}(x)
     \int_\Geps{x} \partial_{n_x} \frac{({n}_y\cdot (\wilde {x}-{y}))^2}{|\wilde x-y|^2} \ds_y.
\end{split}
\end{align}
By the edge-wise continuity of $\trN{v}$, the first two integrals on the right-hand side tend
to integrals with bounded integrands on $\Geps{x}$ for $\wilde x\to x$,
whose limits vanish for $\epsilon\to 0$.
To evaluate the latter two integrals, we assume without loss of generality that
$B_\epsilon(x):=\{\wilde x\in\Omega;\;|\wilde x-x|<\epsilon\}$ is a semicircle
and $\Geps{x}$ a straight line. The general non-polygonal case then follows from
a perturbation argument. Without loss of generality, we further assume that
$B_\epsilon(x)$ lies in the upper half-plane and $\Geps{x}\subset \R\times\{0\}$.
In what follows, we use that the linear polynomial $w(y):=-y_2$
has the trace $\tr{w}=(w,\partial_n w)=(0,1)$ on $\Geps{x}$
and its tensor $\cC\Gradgrad w$ vanishes.
(In the general non-polygonal case, one still has $w(y)=\mathcal{O}(|x-y|^2)$ and
$\partial_n w(y)=1 + \mathcal{O}(|x-y|)$ for $y\in\Geps{x}$.)
We abuse the notation of $\Kpot$ as the double-layer potential on the domain
$B_\epsilon(x)$, and note that the representation formula from Lemma~\ref{la_rep}
for $B_\epsilon(x)$ implies
\begin{align*}
   &1=\partial_{n_x} w(\wilde x)=-\partial_{n_x}\Kpot\trGg{w}(\wilde x)
   =\partial_{n_x} \dual{\trdD{\HG(\wilde x-\cdot))}}{\tr{w}}.
\end{align*}
We calculate this term as before in \eqref{pf_Ka}, though now using \eqref{dual} instead
of \eqref{pre_dual},
\begin{align*}
   &\partial_{n_x} \dual{\trdD{\HG(\wilde x-\cdot))}}{\tr{w}}\\
   &=
     \frac 1{2\pi}
     \int_{\partial B_\epsilon(x)} w(y) \partial_{n_x} \partial_{n_y}\log|\wilde x-y| \ds_y
   + \frac {1-\nu}{4\pi}
     \int_{\partial B_\epsilon(x)} w(y) \partial_{n_x}
     \partial_{t_y} \frac{{t}_y\cdot(\wilde {x}-{y})\, {n}_y\cdot(\wilde {x}-{y})}
                         {|\wilde x-y|^2} \ds_y\\
   &\quad
   - \frac {1+\nu}{4\pi}
     \int_{\partial B_\epsilon(x)} \partial_{n_x} \log|\wilde x-y| \partial_{n_y} w(y) \ds_y
   - \frac {1-\nu}{4\pi}
     \int_{\partial B_\epsilon(x)} \partial_{n_x}
     \frac{({n}_y\cdot (\wilde {x}-{y}))^2}{|\wilde x-y|^2} \partial_{n_y} w(y) \ds_y,
\end{align*}
where the tangential-normal jump terms of $\HG(\wilde x-\cdot)$ at the corners of
$B_\epsilon(x)$ disappear since $w=0$ there.
(In the general non-polygonal case, these jumps vanish for $\epsilon\to 0$.)
We decompose $\partial B_\epsilon(x)$ into $\Geps{x}$ and its complement,
and recall $(w,\partial_n w)|_{\Geps{x}}=(0,1)$. The previous two relations show that,
for $\wilde x\to x$, \begin{align} \label{pf_Kb}
     &-\frac {1+\nu}{4\pi} \int_\Geps{x} \partial_{n_x} \log|\wilde x-y| \ds_y
   - \frac {1-\nu}{4\pi}
     \int_\Geps{x} \partial_{n_x} \frac{({n}_y\cdot (\wilde {x}-{y}))^2}{|\wilde x-y|^2} \ds_y
   \nonumber\\
   &\to 1
   +\frac {1+\nu}{4\pi}
      \int_{\Omega\cap\partial B_\epsilon(x)}
         \partial_{n_x}\log|x-y| \partial_{n_y} w(y) \ds_y
   + \frac {1-\nu}{4\pi}
     \int_{\Omega\cap\partial B_\epsilon(x)} \partial_{n_x}
         \frac{({n}_y\cdot ({x}-{y}))^2}{|x-y|^2} \partial_{n_y} w(y) \ds_y
   \nonumber\\ &\quad
   - \frac 1{2\pi}
      \int_{\Omega\cap\partial B_\epsilon(x)} w(y)
         \partial_{n_x}\partial_{n_y}\log|x-y| \ds_y
   \nonumber\\ &\quad
   - \frac {1-\nu}{4\pi}
     \int_{\Omega\cap\partial B_\epsilon(x)} w(y) \partial_{n_x}
     \partial_{t_y} \frac{{t}_y\cdot({x}-{y})\, {n}_y\cdot({x}-{y})}
                         {|x-y|^2} \ds_y.
\end{align}
By using polar coordinates $(r,\theta)$ centered at $x$, some lengthy calculations reveal
\begin{align*}
   & \int_{\Omega\cap\partial B_\epsilon(x)}
         \partial_{n_x}\log|x-y| \partial_{n_y} w(y) \ds_y
   = -\frac \pi 2,
   &&\hspace{-0.5em}
   \int_{\Omega\cap\partial B_\epsilon(x)} \partial_{n_x}
       \frac{({n}_y\cdot ({x}-{y}))^2}{|x-y|^2} \partial_{n_y} w(y) \ds_y = 0,\\
   & \int_{\Omega\cap\partial B_\epsilon(x)} w(y) \partial_{n_x}\partial_{n_y}\log|x-y| \ds_y
   = \frac \pi 2,
   &&\hspace{-0.5em}
     \int_{\Omega\cap\partial B_\epsilon(x)} w(y) \partial_{n_x}
     \partial_{t_y} \frac{{t}_y\cdot({x}-{y})\, {n}_y\cdot({x}-{y})}
                         {|x-y|^2} \ds_y
   = \frac \pi 2.
\end{align*}
Therefore, returning to \eqref{pf_Kb}, we conclude that
\begin{align*}
     &-\frac \nu{2\pi} \int_\Geps{x} \partial_{n_x} \log|\wilde x-y| \ds_y
   - \frac {1-\nu}{4\pi}
     \int_\Geps{x} \partial_{n_x} \frac{({n}_y\cdot (\wilde {x}-{y}))^2}{|\wilde x-y|^2} \ds_y
   \to 1 - \frac{1+\nu}8 -\frac 14 -\frac{1-\nu}8 = \frac 12
\end{align*}
as $\wilde x\to x$ and $\epsilon\to 0$, that is, recalling \eqref{pf_K_I34},
\[
   I_3(\wilde x,\epsilon) + I_4(\wilde x,\epsilon)
   \to \frac 12\trN{v}(x)\quad (\wilde x\to x, \epsilon\to 0).
\]
This shows that \eqref{pf_Kd} tends to $\trN{v}(x)/2$, which was left to prove the theorem.
\end{proof}

Finally, the representation of $\NIO$ is immediate by the boundedness of $G$ and $\grad G$,
cf.~\eqref{calc}

\begin{theorem}\label{thm_NIO}
Let $f\in L_2(\Omega)$.  At $x\in\Gamma$,
the Newton potential $\NIO f=(\No f,\Nt f)$ has the representation
\begin{align*}
\begin{split} 
   \No f(x)
   &= \int_\Omega f(y) G(x-y)\dy
   = \frac 1{8\pi} \int_\Omega f(y) |x-y|^2\log|x-y|\dy,
\end{split}
\\
\begin{split} 
   \Nt f(x)
   &= \int_\Omega f(y) \partial_{n_x} G(x-y)\dy
   =
   \frac 1{8\pi}\int_\Omega f(y) (2\log|x-y|+1){n}_x\cdot({x}-{y})\dy
   \quad (x\not\in\cN(\Omega)).
\end{split}
\end{align*}
\end{theorem}

\section{Boundary element method} \label{sec_BEM}

\subsection{Variational formulation} \label{sec_BEM_VF}

Given the Dirichlet data $\tr{g}=(\trD{g},\trN{g}) \in H^{3/2,1/2}(\Gamma)$
and $f\in L_2(\Omega)$,
we solve the integral equation \eqref{BIE} with $\tr{u} = \tr{g}$.
By \cite[Theorem~5.2]{CostabelD_96_IBS}, the single-layer operator
$\V:\;H^{-3/2,-1/2}(\Gamma)\to H^{3/2,1/2}(\Gamma)$
is an isomorphism if the domain $\Omega$ is appropriately scaled. 
More precisely, there are only four scaling parameters $c_1,c_2,c_3,c_4>0$ for which the single-layer operator $\boldsymbol{V}_{c_i}$ corresponding to the scaled domain $c_i \Omega$ is not an isomorphism. 
Let us therefore assume that $\Omega$ is appropriately scaled, that is, $\V$ is an isomorphism. 
It is known that $\V$ is positive definite on a closed subspace, see
Proposition~\ref{prop_Vposdef} below.
Since the solution $u$ of \eqref{model} gives rise to a solution
$\tr{m}=\trdD{\bM}$ of \eqref{BIE} for $\bM:=-\cC\Gradgrad u$ and $\V$ is an isomorphism,
the two problems are equivalent, and the representation formula of Lemma~\ref{la_rep} implies that
$u=\Vpot\tr{m}-\Kpot\tr{u}+\Npot f$ solves \eqref{model}.

Let $P_1 = P_1(\Omega)$ be the space of polynomials of degree at most one and 
\[
   H^{-3/2,-1/2}_*(\Gamma)
   :=\{\tr{q}\in H^{-3/2,-1/2}(\Gamma);\; \dual{\tr{q}}{\trGg{v}}=0\ \forall v\in P_1\}.
\]

\begin{prop}[{\cite[Proposition~4.5]{CostabelD_96_IBS}}]\label{prop_Vposdef}
  There is a constant $\alpha>0$ such that
  \[
   \dual{\V\tr{q}}{\tr{q}} \ge \alpha \|\tr{q}\|_{-3/2,-1/2,\Gamma}^2
   \quad\forall \tr{q}\in H^{-3/2,-1/2}_*(\Gamma).
  \]
\end{prop}

Considering the solution $u$ of \eqref{model} and $\bM:=-\Gradgrad u$,
we observe that $\tr{m}:=\trdD{\bM}$ and any $v\in P_1$ satisfy
\begin{align}\label{eq:compatibility}
  \dual{\tr{m}}{\trGg{v}} = \vdual{\div\Div\bM}{v}_\Omega-\vdual{\bM}{\Gradgrad v}_\Omega
  =-\vdual{f}{v}_\Omega
\end{align}
because $\div\Div\bM=-f$ by \eqref{bilap}. 
Therefore, a variational formulation of the integral equation
\eqref{BIE} with given $\tr{u}=\tr{g}\in H^{3/2,1/2}(\Gamma)$ reads:
\emph{Find $\tr{m}\in H^{-3/2,-1/2}(\Gamma)$ and $\lambda\in P_1$ such that}
\begin{subequations} \label{BIE_saddle}
\begin{alignat}{3}
  &\dual{\V\tr{m}}{\tr{q}} + \dual{\tr{q}}{\trGg{\lambda}}
   &&= \dual{\bigl(\K+\frac12\bigr)\tr{g}-\NIO f}{\tr{q}}
   \quad &&\forall\tr{q}\in H^{-3/2,-1/2}(\Gamma),\\
   &\dual{\tr{m}}{\trGg\mu} &&=-\vdual{f}{\mu}_\Omega &&\forall\mu\in  P_1.
\end{alignat}
\end{subequations}

\begin{prop} \label{prop_wellposed}
  Problem~\eqref{BIE_saddle} is well posed. 
\end{prop}
\begin{proof}
  Since all bilinear and linear forms of the mixed system~\eqref{BIE_saddle} are bounded, by the Babu\v{s}ka--Brezzi theorem and Proposition~\ref{prop_Vposdef} it suffices to verify the $\inf$--$\sup$ condition
  \begin{align*}
    \inf_{\mu\in P_1\setminus\{0\}} \sup_{\tr{m}\in H^{-3/2,-1/2}(\Gamma)\setminus\{0\}} \frac{\dual{\tr{m}}{\trGg\mu}}{\|\mu\|_{2,\Omega} \|\tr{m}\|_{-3/2,-1/2,\Gamma}}>0.
  \end{align*}
  For given $\mu\in P_1\setminus\{0\}$, we select $\tr{m} := (0,\mu|_\Gamma,0)\in H^{-3/2,-1/2}(\Gamma)$. Then
\(
    \dual{\tr{m}}{\trGg{\mu}} = \|\mu\|_{\Gamma}^2>0,
\)
  and by the equivalence of $\|\cdot\|_{2,\Omega}$ and $\|\cdot\|_{\Gamma}$ in the
  finite-dimensional space $P_1$, we obtain the desired estimate.
\end{proof}

\begin{remark}
Since Proposition~\ref{prop_Vposdef} holds indeed true independently of the scaling of the domain $\Omega$, the same argument shows that \eqref{BIE_saddle} is also well posed even if $\V$ is not an isomorphism. 
Owing to~\eqref{eq:compatibility}, \eqref{model} is thus always equivalent to the mixed system~\eqref{BIE_saddle}. 
\end{remark}

\subsection{Approximation spaces} \label{sec_BEM_hp}
Let $\cE_h$ denote a partition of $\Gamma$ into non-overlapping open segments $E$ with $h_E:=|E|>0$ and such that there exists $\Gamma_E\in\cE(\Omega)$ with $E\subseteq \Gamma_E$. Set $h:=\max_{E\in\cE_h} h_E$.
With $\cN_h$, we denote all vertices of the mesh $\cE_h$.
For every $E\in\cE_h$ there exists a smooth positively oriented path
$\gamma_E\colon [0,1]\to \overline{E}$ with $|\gamma_E'(t)|\eqsim |E|$ for all $t\in [0,1]$.

We assume that they are all bijective (in particular excluding the case of a smooth
closed curve $\Gamma$ with trivial partition).

Let $P_p(0,1)$ denote the space of polynomials of degree at most $p\in\mathbb{N}_0$ over the unit interval and define the spaces
\begin{alignat*}{2}
  P_p(E) &:= \{v\circ \gamma_E^{-1};\, v\in P_p(0,1)\}, \quad  &&P_1'(E) := \{\partial_t(v\circ\gamma_E^{-1});\, v\in P_1(0,1)\}, \\
  P_p(\cE_h) &:= \{v\in L_2(\Gamma);\, v|_E\in P_p(E), \, E\in\cE_h\}, \quad
  &&P_1'(\cE_h) := \{v\in L_2(\Gamma);\, v|_E\in P_1'(E), \, E\in\cE_h\}.
\end{alignat*}
Note that $P_1'(E)$ coincides with $P_0(E)$ if and only if $|\gamma_E'|$ is constant.
In any case, $\dim(P_1'(E)) = \dim(P_0(E))=1$.

To discretize $H^{-3/2,-1/2}(\Gamma)$ for given polynomial degree $p\in\mathbb{N}_0$, we use the space $X_h$ defined by
\begin{align*}
X_h := \begin{cases} P_0(\cE_h)\times \{0\}\times  \R^{\#\cN_h} \quad(p=0), \\
P_1(\cE_h)\times P_1'(\cE_h)\times \R^{\#\cN(\Omega)} \quad(p=1), \\
P_p(\cE_h)\times P_{p-1}(\cE_h) \times \R^{\#\cN(\Omega)} \quad(p>1).
\end{cases}
\end{align*}
In what follows, we write $X_h\subset H^{-3/2,-1/2}(\Gamma)$, which is to be understood as an embedding. To be specific, 
each $(\trnn{q},\trsf{q},\trJ{q})\in X_h$ defines a functional $\ell\in H^{-3/2,-1/2}(\Gamma)$ by
\begin{align*}
  \ell(\tr{v}) := \begin{cases}-\sum_{E\in\cE_h} \dual{\trnn{q}}{\trN{v}}_E+\sum_{z\in\cN_h} \trJ{q}(z)\trD{v}(z) \quad (p=0), \\
  \sum_{E\in\cE_h}(\dual{\trsf{q}}{\trD{v}}_E- \dual{\trnn{q}}{\trN{v}}_E)+\sum_{z\in\cN(\Omega)} \trJ{q}(z)\trD{v}(z) \quad (p>0),
  \end{cases} 
  \bigl(\tr{v}=(\trD{v},\trN{v})\in H^{3/2,1/2}(\Gamma)\bigr).
\end{align*}
With the usual arguments, one proves that $\ell(\tr{v}) = 0$ for all $\tr{v}$ implies that $(\trnn{q},\trsf{q},\trJ{q}) =0$.
The dimensions of these spaces are
\begin{align*}
  \dim(X_h) &= \begin{cases} \#\cE_h+\#\cN_h = 2\#\cE_h \quad (p=0), \\
  (p+1)\#\cE_h+p\#\cE_h+\#\cN(\Omega) = (2p+1)\#\cE_h + \#\cN(\Omega) \quad (p>0).
  \end{cases} 
\end{align*}
\begin{remark}
Note that $\trJ{q}$ with $(0,0,\trJ{q})\in X_h$ are identified as point distributions. 
For $p=0$, the space $X_h$ has been mentioned and used in~\cite{CostabelD_96_IBS} for numerical experiments. Below, we provide an analysis of its approximation properties.
\end{remark}

To discretize $H^{3/2,1/2}(\Gamma)$, we consider the space
\begin{align*}
  Y_h &= \{(\trD{v},\trN{v})\in P_{p+2+\delta_{p0}}(\cE_h)\times P_{p+1}(\cE_h) ;\, \trD{v}, (\partial_t\trD{v},\trN{v})(t\,\vert\,n)\in C^0(\Gamma)^2\},
\end{align*}
cf.~Remark~\ref{rem_Ggrad}(i).
Here and in the following, $\delta_{jk}\in\{0,1\}$ denotes the Kronecker $\delta$ symbol, and 
$(t\,\vert\,n)$ is the matrix formed by the two columns $t$ and $n$.
To these spaces, we associate the following moments,
\begin{subequations}\label{dofsYhp}
\begin{align}
  \trD{v}(z), (\partial_t\trD{v},\trN{v})({t}\,\vert\,{n})(z), &\qquad z\in \cN_h, \label{dofsYhp:a}\\
  \dual{\trN{v}}{\mu}_E, &\qquad \mu\in P_{p-1}(E),  E\in\cE_h, p\geq 1,  \label{dofsYhp:b}\\
  \dual{\trD{v}}{\mu}_E, &\qquad \mu\in P_{p-2}(E),  E\in\cE_h, p\geq 2.  \label{dofsYhp:c}
\end{align}
\end{subequations}

\begin{prop}\label{prop_dofsYhp}
 For any $p\in\mathbb{N}_0$, $Y_h$ is a subspace of $H^{3/2,1/2}(\Gamma)$
  with dimension $(2p+2+\delta_{p0})\#\cE_h$ and degrees of freedom \eqref{dofsYhp}.
\end{prop}

\begin{proof}
  Let $(\trD{v},\trN{v})\in P_{p+2+\delta_{p0}}(\cE_h)\times P_{p+1}(\cE_h)$ be given. 
  If $\trD{v}$ is continuous at every $z\in\cN_h$ then $\trD{v}\in H^1(\Gamma)$. 
  Note that $(\partial_t \trD{v}|_E,\trN{v}|_E)^\top\cdot{t}|_E$ and
  $(\partial_t \trD{v}|_E,\trN{v}|_E)^\top\cdot{n}|_E$ are smooth on $E\in\cE_h$.
  Therefore, if $(\partial_t \trD{v},\trN{v})(t\,\vert\,n)$ is continuous at every $z\in\cN_h$,
  then $(\partial_t \trD{v},\trN{v})(t\,\vert\,n) \in H^1(\Gamma)^2$ and
  Lemma~\ref{lem_isomorphic} show that $Y_h\subset H^{3/2,1/2}(\Gamma)$. 

  Note that $\dim(P_{p+2+\delta_{p0}}(\cE_h)\times P_{p+1}(\cE_h)) = (2p+5+\delta_{p0})\#\cE_h$.
  By the continuity requirements (three conditions per vertex) it follows that
  $\dim(Y_h) = (2p+2+\delta_{p0})\#\cE_h$.
  This is equal to the number of moments~\eqref{dofsYhp},
   namely $3\#\cN_h + (2p-1+\delta_{p0})\#\cE_h$.

  Let $(\trD{v},\trN{v})\in Y_h$ and suppose that all the moments~\eqref{dofsYhp} vanish. Let $E\in\cE_h$. 
  We stress that $\trD{v}(z) = 0$, $(\partial_t\trD{v},\trN{v})^\top\cdot{t}(z)=0=(\partial_t\trD{v},\trN{v})^\top\cdot{n}(z)$ imply 
  that $\trD{v}(z) = 0$ as well as $\partial_t \trD{v}(z) = 0 = \trN{v}(z)$ for all $z\in\partial E$.
  Note that $v_n|_E\in P_{p+1}(E)$. So, $\trN{v}(z) =0$ for all $z\in\partial E$ and $\dual{\trN{v}}{\mu}_E = 0$ for all $\mu\in P_{p-1}(E)$ imply that $\trN{v}|_E = 0$.
  A similar argumentation shows that $\trD{v}|_E = 0$. Since $E\in\cE_h$ was arbitrary, we conclude that if all moments~\eqref{dofsYhp} vanish, then $(\trD{v},\trN{v}) = 0$.
\end{proof}

The canonical interpolation operator which maps into $Y_h$ and interpolates in~\eqref{dofsYhp} is denoted by $I_h$. Note that $I_h(\trD{v},\trN{v})$ is well defined if $v_0, (\partial_t\trD{v},\trN{v})^\top\cdot{t}, (\partial_t\trD{v},\trN{v})^\top\cdot{n}\in C^0(\Gamma)$.
Below, we study its approximation properties.

\begin{remark}
  Note that if $\Gamma$ is polygonal, then $Y_h$ for ($p=0$) $p=1$ is the trace space of the (reduced) Hsieh--Clough--Tocher element~\cite{Ciarlet_78}.
  Further, note that for $p=0$ each $(\trD{v},\trN{v})\in Y_h$ satisfies $\trD{v}|_E \in P_3(E)$ but $\trN{v}|_E\in P_1(E)$ for $E\in\cE_h$. 
  This ``imbalance'' in the polynomial orders is required to ensure the continuity constraints in the space. 
\end{remark}

\subsection{Approximation estimates in $H^{3/2,1/2}(\Gamma)$} \label{sec_BEM_est_dir}
Before we discuss the approximation properties of $I_h$, we consider the following local version.
\begin{lemma}\label{lem_NodIntLocal}
  Let $p\in\mathbb{N}_0$, $E\in\cE_h$. 
  Given $(\trD{v},\trN{v})\in C^1(\overline E)\times C^0(\overline E)$,
  define $(\trD{w},\trN{w})\in P_{p+2+\delta_{p0}}(E)\times P_{p+1}(E)$ by
  \begin{align*}
    \trD{w}(z) = \trD{v}(z), \quad (\partial_t \trD{w},\trN{w})({t} \,\vert\, {n})(z) = (\partial_t \trD{v},\trN{v})({t}\,\vert\, {n})(z) &\quad \forall z\in \partial E, \\
    \dual{\trN{w}}\mu_E = \dual{\trN{v}}\mu_E &\quad \forall \mu\in P_{p-1}(E) \text{ if }p\geq 1, \\
    \dual{\trD{w}}\mu_E = \dual{\trD{v}}\mu_E &\quad \forall \mu\in P_{p-2}(E) \text{ if }p\geq 2.
  \end{align*}
  Then,  for $2\leq s\leq p+3+\delta_{p0}$,
  \begin{align*}
    \|\trD{v}-\trD{w}\|_E + h_E \|\partial_t(\trD{v}-\trD{w})\|_E + h_E^2 \|\partial_t^2(\trD{v}-\trD{w})\|_E &\lesssim h_E^{s}\|\trD{v}\|_{s,E} \quad\forall \trD{v}\in H^s(E),
\intertext{and, for $1\leq s\leq p+2$,}
    \|\trN{v}-\trN{w}\|_E + h_E\|\partial_t(\trN{v}-\trN{w})\|_E &\lesssim h_E^s\|\trN{v}\|_{s,E} \quad\forall \trN{v}\in H^s(E).
  \end{align*}
\end{lemma}
\begin{proof}
  Note that $(\partial_t \trD{w},\trN{w})({t} \,\vert\, {n})(z) = (\partial_t \trD{v},\trN{v})({t}\,\vert\, {n})(z)$ implies that $\partial_t\trD{w}(z) = \partial_t\trD{v}(z)$ for $z\in\partial E$. Hence, together with $\trD{w}(z) = \trD{v}(z)$, one sees that $\trD{w}$ interpolates $\trD{v}$ in the endpoints.
  By standard approximation estimates we conclude the first assertion. The second one follows by a similar argumentation. 
\end{proof}

\begin{theorem}\label{thm_NodInt}
  Let $p\in\mathbb{N}_0$.
  If $\tr{v}=(\trD{v},\trN{v})\in H^{3/2,1/2}(\Gamma)$ with $\trD{v}\in H^{p+3}(\cE_h)$, $\trN{v}\in H^{p+2}(\cE_h)$, then
  \begin{align*}
    \|(1-I_h)\tr{v}\|_{3/2,1/2,\Gamma} \lesssim h^{p+3/2}(\|\trD{v}\|_{p+3,\cE_h} + \|\trN{v}\|_{p+2,\cE_h}).
  \end{align*}
\end{theorem}
\begin{proof}
  Since $\trD{v} \in H^{p+3}(\cE_h)$, $\trN{v}\in H^{p+2}(\cE_h)$, the conditions $(\partial_t\trD{v},\trN{v})^\top\cdot{t} \in H^{1/2}(\Gamma)$ and $(\partial_t\trD{v},\trN{v})^\top\cdot{n} \in H^{1/2}(\Gamma)$ from Lemma~\ref{lem_isomorphic} imply that $(\partial_t\trD{v},\trN{v})^\top\cdot{t} \in C^0(\Gamma)$ and $(\partial_t\trD{v},\trN{v})^\top\cdot{n}\in C^0(\Gamma)$.
  Therefore, $I_h\tr{v}$ is well defined. 
  In fact, we even have that $(\partial_t\trD{v},\trN{v})^\top\cdot{t} \in H^1(\Gamma)$ and $(\partial_t\trD{v},\trN{v})^\top\cdot{n} \in H^1(\Gamma)$.

  Let us write $I_h\tr{v} = (\trD{w},\trN{w})$, 
  and recall definition \eqref{tnorm} of $\tnorm{\cdot}_{3/2,1/2,\Gamma}$.
  Since $\tnorm{\cdot}_{3/2,1/2,\Gamma}$ and $\|\cdot\|_{3/2,1/2,\Gamma}$ are
  equivalent norms in $H^{3/2,1/2}(\Gamma)$ by Lemma~\ref{lem_isomorphic}, it is enough to bound
  \begin{align}\label{thm_NodInt:proof1}
    \|\trD{v}-\trD{w}\|_{1,\Gamma} + \|(\partial_t(\trD{v}-\trD{w}),\trN{v}-\trN{w})^\top\cdot{t}\|_{1/2,\Gamma}
    + \|(\partial_t(\trD{v}-\trD{w}),\trN{v}-\trN{w})^\top\cdot{n}\|_{1/2,\Gamma}.
  \end{align}
  The first term is estimated by applying Lemma~\ref{lem_NodIntLocal}, which results in
  \begin{align*}
    \|\trD{v}-\trD{w}\|_{1,\Gamma} \lesssim h^{p+2}\|\trD{v}\|_{p+3,\cE_h} \lesssim h^{p+3/2}\|\trD{v}\|_{p+3,\cE_h}.
  \end{align*}
  Next, we consider the second term of~\eqref{thm_NodInt:proof1}. 
  By our previous considerations we have that $(\partial_t(\trD{v}-\trD{w}),\trN{v}-\trN{w})^\top\cdot{t} \in H^1(\Gamma)$. 
  The interpolation estimate then gives
  \begin{align}\label{thm_NodInt:proof2}
    \begin{split}
    &\|(\partial_t(\trD{v}-\trD{w}),\trN{v}-\trN{w})^\top\cdot{t}\|_{1/2,\Gamma}
    \\
    &\qquad\lesssim 
    \|(\partial_t(\trD{v}-\trD{w}),\trN{v}-\trN{w})^\top\cdot{t}\|_{\Gamma}^{1/2}
    \|(\partial_t(\trD{v}-\trD{w}),\trN{v}-\trN{w})^\top\cdot{t}\|_{1,\Gamma}^{1/2}.
  \end{split}
  \end{align}
  Noting that $|{t}|=1$ and applying Lemma~\ref{lem_NodIntLocal}, we obtain
  \begin{align*}
    \|(\partial_t(\trD{v}-\trD{w}),\trN{v}-\trN{w})^\top\cdot{t}\|_E 
    &\leq \|\partial_t(\trD{v}-\trD{w})\|_E + \|\trN{v}-\trN{w}\|_E
    \\
    &\lesssim h_E^{p+2}\|\trD{v}\|_{p+3,E} + h_E^{p+2}\|\trN{v}\|_{p+2,E}.
  \end{align*}
  Furthermore, the triangle inequality and product rule show for $E\in\cE_h$ that
  \begin{align*}
    \|\partial_t\big((\partial_t(\trD{v}-\trD{w}),\trN{v}-\trN{w})^\top\cdot{t}\big)\|_E
    &\leq \|\partial_t^2(\trD{v}-\trD{w}){t}_1\|_E+\|\partial_t(\trD{v}-\trD{w})\partial_t {t}_1\|_E
    \\
    &\qquad + \|\partial_t(\trN{v}-\trN{w}){t}_2\|_E+\|(\trN{v}-\trN{w})\partial_t {t}_2\|_E.
  \end{align*}
  Using that $\|\partial_t{t}\|_{\infty,E}\lesssim 1$ together with the approximation properties from Lemma~\ref{lem_NodIntLocal}, we further obtain
  \begin{align*}
    \|\partial_t\big((\partial_t(\trD{v}-\trD{w}),\trN{v}-\trN{w})^\top\cdot{t}\big)\|_E &\lesssim (h_E^{p+1}+h_E^{p+2})\|\trD{v}\|_{p+3,E}
    + (h_E^{p+1}+h_E^{p+2})\|\trN{v}\|_{p+2,E}
    \\
    &\lesssim h_E^{p+1}(\|\trD{v}\|_{p+3,E}+\|\trN{v}\|_{p+2,E}).
  \end{align*}
  In view of~\eqref{thm_NodInt:proof2}, we have that
  \begin{align*}
    \|(\partial_t(\trD{v}-\trD{w}),\trN{v}-\trN{w})^\top\cdot{t}\|_{1/2,\Gamma}
    &\lesssim h^{p/2+1}(\|\trD{v}\|_{p+3,\cE_h} + \|\trN{v}\|_{p+2,\cE_h})^{1/2}
    \\
    &\qquad\times h^{p/2+1/2}(\|\trD{v}\|_{p+3,\cE_h} + \|\trN{v}\|_{p+2,\cE_h})^{1/2}.
  \end{align*}
  The third term in~\eqref{thm_NodInt:proof1} can be estimated with the same arguments.
  We omit further details. 

  Overall, we conclude that~\eqref{thm_NodInt} can be bounded as claimed.
\end{proof}

An alternative to interpolation operator $I_h$ is the orthogonal projection
$\Pi_h \colon H^{1,0}(\Gamma) \to Y_h$ on
$H^{1,0}(\Gamma) := H^1(\Gamma)\times L_2(\Gamma)$ with canonical norm $\|\cdot\|_{1,0,\Gamma}$.

\begin{cor} \label{cor_Pihp}
  Let $p\in\mathbb{N}_0$.
  If $\tr{v}=(\trD{v},\trN{v})\in H^{3/2,1/2}(\Gamma)$ with $\trD{v}\in H^{p+3}(\cE_h)$, $\trN{v}\in H^{p+2}(\cE_h)$, then
  \begin{align*}
    \|(1-\Pi_h)\tr{v}\|_{3/2,1/2,\Gamma} \lesssim h^{p+3/2}(\|\trD{v}\|_{p+3,\cE_h} + \|\trN{v}\|_{p+2,\cE_h}).
  \end{align*}
\end{cor}

\begin{proof}
  We need the inverse estimate
  \begin{align*}
    \|\tr{w}\|_{3/2,1/2,\Gamma} 
    \lesssim h^{-1/2} \|\tr{w}\|_{1,0,\Gamma}
    \quad \forall \tr{w}=(\trD{w},\trN{w}) \in Y_h.
  \end{align*}
  It is a consequence of standard inverse estimates and
  arguments used
  in the proof of Theorem~\ref{thm_NodInt} to bound $\tnorm{\cdot}_{3/2,1/2,\Gamma}$,
  cf.~\eqref{thm_NodInt:proof1}.
  This implies
  \begin{align*}
    \|(1-\Pi_h)\tr{v}\|_{3/2,1/2,\Gamma} 
    &= \|(1-\Pi_h)(1-I_h)\tr{v}\|_{3/2,1/2,\Gamma} 
    \\
    &\le \|(1-I_h)\tr{v}\|_{3/2,1/2,\Gamma} + \|\Pi_h(1-I_h)\tr{v}\|_{3/2,1/2,\Gamma} 
    \\
    &\lesssim \|(1-I_h)\tr{v}\|_{3/2,1/2,\Gamma} + h^{-1/2} \|\Pi_h(1-I_h)\tr{v}\|_{1,0,\Gamma} 
      \\
      &\le \|(1-I_h)\tr{v}\|_{3/2,1/2,\Gamma} + h^{-1/2} \|(1-I_h)\tr{v}\|_{1,0,\Gamma}.
  \end{align*}

The arguments in the proof of Theorem~\ref{thm_NodInt:proof1} also show that
\begin{align*}
  \|(1-I_h)\tr{v}\|_{1,0,\Gamma} \lesssim h^{p+2}(\|\trD{v}\|_{p+3,\cE_h} + \|\trN{v}\|_{p+2,\cE_h}).
\end{align*}
Therefore, the approximation result from Theorem~\ref{thm_NodInt} leads to the assertion.

\end{proof}

\subsection{Approximation estimates in $H^{-3/2,-1/2}(\Gamma)$} \label{sec_BEM_est_neu}

In the following, let $\npi{p}{E}\,\colon L_2(E)\to P_{p}(E)$ for $E\in\cE_h$ and
$\npi[0]{p}{\cE_h(\Gamma')}\,\colon L_2(\Gamma')\to P_p(\cE_h(\Gamma'))\cap C^0(\Gamma')$
be orthogonal projections.
Here, $\cE_h(\Gamma')\subset \cE_h$ denotes the mesh of $\Gamma'\in\cE(\Omega)$.
The next result follows from standard approximation estimates, see, e.g.,~\cite[Corollary~22.9]{ErnGuermondI} and~\cite[Theorem~10.2]{Steinbach_08_NAM}.

\begin{lemma}\label{lem_localApp}
  Let $p\in\mathbb{N}_0$, $E\in\cE_h$, and $\Gamma'\in \cE(\Omega)$. The estimates
  \begin{alignat*}{2}
    \|(1-\npi{p}{E})w\|_E
    &\lesssim h_E^s \|w\|_{s,E} &\quad&\forall w\in H^s(E),\\
    \|(1-\npi[0]{p}{\cE_h(\Gamma')})w\|_{\Gamma'}
    &\lesssim h^s \|w\|_{s,\cE_h(\Gamma')}
    &\quad&\forall w\in H^s(\cE_h(\Gamma'))\cap C^0(\overline{\Gamma'}),
    \\
    \|(1-\npi[0]{p}{\cE_h(\Gamma')})w\|_{\Gamma'} &\lesssim h^s \|w\|_{s,\Gamma'} &\quad&\forall w\in H^s(\Gamma')
  \end{alignat*}
  hold for $0\leq s \leq p+1$.
\end{lemma}

\begin{theorem}\label{thm_BestApp}
  Let $\tr{m} = (\trnn{m},\trsf{m},\trJ{m})\in H^{-3/2,-1/2}(\Gamma)$ and $p\in\mathbb{N}_0$ be given. 
  If $\trnn{m} \in H^{p+1}(\cE_h)$ and $\trsf{m}\in H^{p}(\cE_h)$, then 
  \begin{align*}
    \min_{\tr{q}\in X_h}\|\tr{m}-\tr{q}\|_{-3/2,-1/2,\Gamma} \lesssim h^{p+3/2}(\|\trnn{m}\|_{p+1,\cE_h} + \|\trsf{m}\|_{p,\cE_h}). 
  \end{align*}
\end{theorem}

\begin{proof}
  {\bf The case $p\geq 1$.}
  Given any $\tr{q}\in X_h$, we write with $\tr{v}=(\trD{v},\trN{v})\in H^{3/2,1/2}(\Gamma)$
  \begin{align}\label{errtraceformula}
    \dual{\tr{m}-\tr{q}}{\tr{v}} = \sum_{E\in\cE_h} (\dual{\trsf{m}-\trsf{q}}{\trD{v}}_E - \dual{\trnn{m}-\trnn{q}}{\trN{v}}_E) + \sum_{z\in\cN(\Omega)} (\trJ{m}-\trJ{q})(z)\trD{v}(z).
  \end{align}
  We choose $\trnn{q}|_E := \npi{p}{E}\,\trnn{m}$ for all $E\in\cE_h$. Then, for any $E\in\cE_h$, Lemma~\ref{lem_localApp} yields
  \begin{align*}
    |\dual{\trnn{m}-\trnn{q}}{\trN{v}}_E| &= |\dual{(1-\npi{p}{E})\trnn{m}}{(1-\npi{p}{E})\trN{v}}_E| 
    \lesssim h_E^{p+1}\|\trnn{m}\|_{p+1,E} h_E^{1/2}\|\trN{v}\|_{1/2,E}.
  \end{align*}
  By the subadditivity of Sobolev norms and estimate \eqref{trestH2}, we have that
  \begin{align*}
    \sum_{E\in\cE_h}\|\trN{v}\|_{1/2,E}^2 \lesssim \sum_{\Gamma'\in\cE(\Omega)} \|\trN{v}\|_{1/2,\Gamma'}^2 \lesssim \|\tr{v}\|_{3/2,1/2,\Gamma}^2
  \end{align*}
  and, thus, 
  \begin{align} \label{pf_BestApp1}
    -\sum_{E\in\cE_h} \dual{\trnn{m}-\trnn{q}}{\trN{v}}_E \lesssim h^{p+3/2} \|\trnn{m}\|_{p+1,\cE_h} \|\tr{v}\|_{3/2,1/2,\Gamma}.
  \end{align}
  For the remaining two terms in~\eqref{errtraceformula}, we first consider the case $p=1$. 
  Let $\Gamma'\in\cE(\Omega)$ and note that there exists $\psi_{\Gamma'}\in H^{1}(\Gamma')$ with $\partial_t\psi_{\Gamma'} = \trsf{m}|_{\Gamma'}$.
  We choose $\trsf{q}|_{\Gamma'} := \partial_t(\npi[0]{1}{\cE_h(\Gamma')}\psi_{\Gamma'})$ and observe via integration by parts that
  \begin{align*}
    \sum_{\Gamma'\in\cE(\Omega)}\dual{\trsf{m}-\trsf{q}}{\trD{v}}_{\Gamma'}
    &=-\sum_{\Gamma'\in\cE(\Omega)}
       \dual{(1-\npi[0]{1}{\cE_h(\Gamma')})\psi_{\Gamma'}}{\partial_t \trD{v}}_{\Gamma'}
     + \bigl((1-\npi[0]{1}{\cE_h(\Gamma')})\psi_{\Gamma'}\bigr)\trD{v}\Big|_{\partial \Gamma'}.
    \\
    &=-\sum_{\Gamma'\in\cE(\Omega)}
       \dual{(1-\npi[0]{1}{\cE_h(\Gamma')})\psi_{\Gamma'}}{\partial_t \trD{v}}_{\Gamma'}
     + \sum_{z\in\cN(\Omega)} \alpha(z)\trD{v}(z),
  \end{align*}
  where $\alpha(z)=[(1-\npi[0]{1}{\cE_h(\Gamma')})\psi_{\Gamma'}](z)\in \R$ for $z\in \cN(\Omega)$. 
  Finally, we choose $\trJ{q} := \trJ{m}+\alpha$. It then follows that
  \begin{align*}
    \sum_{E\in\cE_h} \dual{\trsf{m}-\trsf{q}}{\trD{v}}_{E} + \sum_{z\in\cN(\Omega)} (\trJ{m}-\trJ{q})(z)\trD{v}(z)
    = -\sum_{\Gamma'\in\cE(\Omega)}\dual{(1-\npi[0]{1}{\cE_h(\Gamma')})\psi_{\Gamma'}}{\partial_t \trD{v}}_{\Gamma'}.
  \end{align*}
  We further estimate the last term with Lemma~\ref{lem_localApp} as
  \begin{align*}
    &\sum_{\Gamma'\in\cE(\Omega)}|\dual{(1-\npi[0]{1}{\cE_h(\Gamma')})\psi_{\Gamma'}}{\partial_t \trD{v}}_{\Gamma'}|
    =  \sum_{\Gamma'\in\cE(\Omega)}
        |\dual{(1-\npi[0]{1}{\cE_h(\Gamma')})\psi_{\Gamma'}}
              {(1-\npi[0]{1}{\cE_h(\Gamma')})\partial_t \trD{v}}_{\Gamma'}|
    \\
    &\qquad\lesssim \sum_{\Gamma'\in\cE(\Omega)}
       \|(1-\npi[0]{1}{\cE_h(\Gamma')})\psi_{\Gamma'}\|_{\Gamma'}
       \|(1-\npi[0]{1}{\cE_h(\Gamma')})\partial_t\trD{v}\|_{\Gamma'}
    \\
    &\qquad\lesssim \sum_{\Gamma'\in\cE(\Omega)}
      h^{2}\|\psi_{\Gamma'}\|_{p+1,\cE_h(\Gamma')}h^{1/2}\|\partial_t\trD{v}\|_{1/2,\Gamma'}
    \leq h^{5/2}\|\trsf{m}\|_{1,\cE_h}\|\partial_t\trD{v}\|_{1/2,\cE(\Omega)}.
  \end{align*}
  Noting that $\|\partial_t\trD{v}\|_{1/2,\cE(\Omega)}\lesssim \|\tr{v}\|_{3/2,1/2,\Gamma}$
  by~\eqref{trestH2}, we conclude with \eqref{pf_BestApp1} that, for $p=1$,
  \begin{align}\label{dualestp1}
    |\dual{\tr{m}-\tr{q}}{\tr{v}}| \lesssim h^{5/2}(\|\trnn{m}\|_{p+1,\cE_h} + \|\trsf{m}\|_{p,\cE_h})\|\tr{v}\|_{3/2,1/2,\Gamma}.
  \end{align}
  Next, consider $p>1$. We choose $\trsf{q}|_E := \npi{p-1}{E}\trsf{m}$ and $\trJ{q} := \trJ{m}$.
  Then, Lemma~\ref{lem_localApp} yields
  \begin{align}\label{dualestp}
  \begin{split}
    \sum_{E\in\cE_h} |\dual{\trsf{m}-\trsf{q}}{\trD{v}}_{E}| + \sum_{z\in\cN(\Omega)} |(\trJ{m}-\trJ{q})(z)\trD{v}(z)|
    &= \sum_{E\in\cE_h} |\dual{(1-\npi{p-1}{E})\trsf{m}}{(1-\npi{p-1}{E})\trD{v}}_{E}| \\
    &\lesssim \sum_{E\in\cE_h} h^{p}\|\trsf{m}\|_{p,E}
    h^{3/2} \|\trD{v}\|_{3/2,E}.
  \end{split}
  \end{align}
  By Lemma~\ref{lem_duality}, estimates~\eqref{dualestp1} and~\eqref{dualestp}
  prove the statement for $p\geq 1$.
  
  {\bf The case $p=0$.} In this case, instead of~\eqref{errtraceformula}, we have the form
  \begin{align*}
    \dual{\tr{m}-\tr{q}}{\tr{v}}= \sum_{E\in\cE_h}(\dual{\trsf m}{\trD{v}}_E-\dual{\trnn{m}-\trnn{q}}{\trN{v}}_E) + \sum_{z\in\cN_h} (\trJ{m}-\trJ{q})(z)\trD{v}(z).
  \end{align*}
  The term $\sum_{E\in\cE_h}\dual{\trnn{m}-\trnn{q}}{\trN{v}}_E$ can be estimated as before.
  In a similar spirit as for the case $p=1$, we choose $\psi_E\in H^1(E)$ for any $E\in\cE_h$ such that $\partial_t\psi_E = \trsf{m}|_E$.
  Note that $\partial_t\psi_E = \partial_t(\psi_E-\npi{0}{E}\,\psi_E)$.
  Using this observation, we obtain with integration by parts
  \begin{align*}
    \sum_{E\in\cE_h} \dual{\trsf{m}}{\trD{v}}_E &= -\sum_{E\in\cE_h} \dual{(1-\npi{0}{E})\psi_E}{\partial_t \trD{v}}_E + \big((1-\npi{0}{E})\psi_E\big)\trD{v}|_{\partial E} 
    \\
    &= -\sum_{E\in\cE_h} \dual{(1-\npi{0}{E})\psi_E}{\partial_t \trD{v}}_E + \sum_{z\in\cN_h} \alpha(z)\trD{v}(z),
  \end{align*}
  where $\alpha(z) = [(1-\npi{0}{E})\psi_E](z)$, $z\in\cN_h$. 
  This leads to choosing $\trJ{q}(z) := \alpha(z) $ for $z\in\cN_h\setminus\cN(\Omega)$
  and $\trJ{q}(z) := \trJ{m}(z)+\alpha(z)$ for $z\in\cN(\Omega)$. Overall, we conclude with Lemma~\ref{lem_localApp} that
  \begin{align*}
    &\sum_{E\in\cE_h} \dual{\trsf{m}}{\trD{v}}_E + \sum_{z\in\cN(\Omega)} \trJ{m}(z)\trD{v}(z) - \sum_{z\in\cN_h} \trJ{q}(z)\trD{v}(z)
    = -\sum_{E\in\cE_h} \dual{(1-\npi{0}{E})\psi_E}{\partial_t \trD{v}}_E
    \\
    &\quad= -\sum_{E\in\cE_h} \dual{(1-\npi{0}{E})\psi_E}{(1-\npi{0}{E})\partial_t \trD{v}}_E
    \lesssim  h^{3/2}\sum_{E\in\cE_h}\|\psi_E\|_{1,E}\|\partial_t\trD{v}\|_{1/2,E}
    \lesssim  h^{3/2}\|\trsf{m}\|_{\Gamma}\|\tr{v}\|_{3/2,1/2,\Gamma}.
  \end{align*}
  An application of Lemma~\ref{lem_duality} finishes the proof.
\end{proof}

\subsection{Discretization and convergence} \label{sec_BEM_conv}
A direct discretization of~\eqref{BIE_saddle} reads:
\emph{Find $\tr{\widetilde{m}}_h\in X_h$ and $\widetilde\lambda_h\in P_1$ such that}
\begin{subequations} 
\begin{alignat*}{3}
  &\dual{\V\tr{\widetilde{m}}_h}{\tr{q}_h} + \dual{\tr{q}_h}{\trGg{\widetilde\lambda_h}}
  &&= \dual{\bigl(\K+\frac12\bigr)\tr{g}-\NIO f}{\tr{q}_h}
  \quad &&\forall\tr{q}_h\in X_h,\\
  &\dual{\tr{\widetilde{m}}_h}{\trGg{\mu}} &&=-\vdual{f}{\mu}_\Omega &&\forall\mu\in  P_1.
\end{alignat*}
\end{subequations}
The latter formulation is not feasible in practice as it requires to compute the
integrals on the right-hand side for given data $\tr{g}$ and $f$.
We propose to replace $\tr{g}$ and $f$ by suitable approximations $\tr{g}_h$ and $f_h$
that allow to compute $\dual{\bigl(\K+\frac 12\bigr)\tr{g}_h-\NIO f_h}{\tr{q}_h}$.
This yields the boundary element method:
\emph{Find $\tr{m}_h\in X_h$ and $\lambda_h\in P_1$ such that}
\begin{subequations} \label{BIE_saddle:disc}
\begin{alignat}{3}
  &\dual{\V\tr{m}_h}{\tr{q}_h} + \dual{\tr{q}_h}{\trGg{\lambda_h}}
  &&= \dual{\bigl(\K+\frac 12\bigr)\tr{g}_h-\NIO f_h}{\tr{q}_h}
  \quad &&\forall\tr{q}_h\in X_h, \label{BIE_saddle:disc:a} \\
  &\dual{\tr{m}_h}{\trGg{\mu}} &&=-\vdual{f_h}{\mu}_\Omega
  &&\forall\mu\in  P_1.\label{BIE_saddle:disc:b}
\end{alignat}
\end{subequations}

\begin{theorem}\label{thm_BEM}
  Let $(\tr{m},\lambda)\in H^{-3/2,-1/2}(\Gamma)\times P_1$ be the solution of~\eqref{BIE_saddle}.
  If $p\leq 1$, assume that $\cN_h$ contains at least three points that are not colinear.
  Problem~\eqref{BIE_saddle:disc} admits a unique solution
  $(\tr{m}_h,\lambda_h)\in X_h\times  P_1$, and it satisfies
  \begin{align*}
    \|\tr{m}-\tr{m}_h\|_{-3/2,-1/2,\Gamma} \lesssim \min_{\tr{q}_h\in X_h} \|\tr{m}-\tr{q}_h\|_{-3/2,-1/2,\Gamma} + \|\tr{g}-\tr{g}_h\|_{3/2,1/2,\Gamma}
    +\|f-f_h\|_{-2,\Omega}.
  \end{align*}
\end{theorem}
\begin{proof}
  First, let us establish unique solvability. The bilinear form $\dual{\V\cdot}\cdot$ is positive definite on the kernel $\{\tr{q}_h\in X_h;\, \dual{\tr{q}_h}{\trGg{\mu}} = 0 \,\forall \mu\in  P_1\} \subset H_*^{-3/2,-1/2}(\Gamma)$. Furthermore, the bilinear forms on the left-hand side and the functional on the right-hand side of~\eqref{BIE_saddle:disc} are continuous. By the Babu\v{s}ka--Brezzi theorem, it suffices to prove the discrete $\inf$-$\sup$ stability
  for $\dual{\tr{m}_h}{\trGg{\mu}}$ to guarantee unique solvability. 
  We distinguish between the three cases $p=0$, $p=1$, and $p>1$.
  
  Let $p=0$. By assumption, there are (at least) three non-colinear vertices $z_1,z_2,z_3\in\cN_h$. 
  Any $\mu\in  P_1$ is uniquely determined by its values in the vertices $z_1,z_2,z_3$, say, $\alpha_1,\alpha_2,\alpha_3$. 
  Choose $\tr{m}_h := (0,0,\trJ{q})\in X_h$ where $\trJ{q}(z_j) := \alpha_j$ for $j=1,2,3$,
  and $\trJ{q}(z) = 0$ for all other vertices. 
  Then, $\dual{\tr{m}_h}{\trGg{\mu}} = \sum_{j=1}^3 \alpha_j^2>0$ for $\mu\neq 0$. 
  
  Let $p=1$ and $\mu\in P_1$. Suppose that $\dual{\tr{m}_h}{\trGg{\mu}}= 0$ for all $\tr{m}_h\in X_h$. 
  We show that this implies $\mu=0$. For any $E\in\cE_h$ choose $\tr{m}_h = (\chi_E,0,0)\in X_h$ with $\chi_E\in P_0(\cE_h)$ denoting the indicator function on $E$. Then, $0=\dual{\tr{m}_h}{\trGg{\mu}} = -\dual{1}{\partial_n \mu}_E$ for all $E\in\cE_h$ implies that $\mu$ is constant by the assumption that $\cN_h$ contains at least three non-colinear nodes.
  Fix $E\in\cE_h$ and choose $\tr{m}_h := (0,\partial_t w,0)$ where $w\in P_1(\cE_h)$ with $w|_{E'} = 0$ for $E'\in\cE_h\setminus\{E\}$, and $w(z^-) = 0$, $w(z^+) = 1$, where $z^\pm$ denotes the two vertices of $E$ (in positive direction). 
  Then, $0=\dual{\tr{m}_h}{\trGg\mu} = \dual{\partial_t w}{\mu}_E = \mu(z^+)$
  implies that $\mu=0$.

  Let $p>1$. Given $\mu\in P_1$ set $\tr{m}_h:=(0,\mu|_\Gamma,0)$ and note that
  $\tr{m}_h \in X_h$. Then, $\dual{\tr{m}_h}{\trGg{\mu}} = \|\mu\|_{\Gamma}^2$
  implies the required inf-sup stability.
  
  To show the error estimate, we consider the auxiliary problem: \emph{Find $\tr{\widehat{m}}\in H^{-3/2,-1/2}(\Gamma)$ and $\widehat\lambda\in P_1$ such that}
  \begin{subequations} 
  \begin{alignat*}{3}
    &\dual{\V\tr{\widehat{m}}}{\tr{q}} + \dual{\tr{q}}{\trGg{\widehat\lambda}}
     &&= \dual{\bigl(\K+\frac 12\bigr)\tr{g}_h-\NIO f_h}{\tr{q}}
     \quad &&\forall\tr{q}\in H^{-3/2,-1/2}(\Gamma),\\
     &\dual{\tr{\widehat{m}}}{\trGg{\mu}} &&=-\vdual{f_h}{\mu}_\Omega &&\forall\mu\in  P_1.
  \end{alignat*}
  \end{subequations}
  Let $(\tr{\widehat m},\widehat\lambda)\in H^{-3/2,-1/2}(\Gamma)$ denote the unique solution of the latter problem and let $(\tr{m}_h,\lambda_h)\in X_h\times P_1$ denote the unique solution of~\eqref{BIE_saddle:disc}.  
  Note that $(\tr{m}_h,\lambda_h)$ is the Galerkin approximation of $(\tr{\widehat m},\widehat\lambda)$. 
  The theory of mixed methods shows that
  \begin{align*}
    \|\tr{\widehat m}-\tr{m}_h\|_{-3/2,-1/2,\Gamma} \lesssim \|\tr{\widehat m}-\tr{q}\|_{-3/2,-1/2,\Gamma} 
    \leq \|\tr{m}-\tr{q}\|_{-3/2,-1/2,\Gamma} + \|\tr{m}-\tr{\widehat m}\|_{-3/2,-1/2,\Gamma}
  \end{align*}
  for all $\tr{q}\in X_h$. By continuous dependence on the data of the continuous problem, we conclude that
  \begin{align*}
    \|\tr{m}-\tr{\widehat m}\|_{-3/2,-1/2,\Gamma}
    \lesssim \|\tr{g}-\tr{g}_h\|_{3/2,1/2,\Gamma} +\|f-f_h\|_{-2,\Omega}.
  \end{align*}
  The proof is finished by combining the latter two estimates and the triangle inequality.
\end{proof}

Besides the choice $\tr{g}_h = I_h\tr{g}$, an alternative choice is $\tr{g}_h = \Pi_h\tr{g}$ with the orthogonal projection
$\Pi_h \colon H^{1,0}(\Gamma) \to Y_h$ on $H^{1,0}(\Gamma) := H^1(\Gamma)\times L_2(\Gamma)$ with canonical norm $\|\cdot\|_{1,0,\Gamma}$.
A natural choice of $f_h$ is a piecewise polynomial projection of $f$.
To be specific, let $\mathcal{T}_h$ be a shape-regular triangulation
of $\Omega$ (with possibly curvilinear elements at the boundary) with mesh-size
of order $h$, and let $\Pi_h^\Omega:\;L_2(\Omega)\to P_q(\mathcal{T}_h)$ denote
the $L_2$-projection onto piecewise polynomials of degree $q\in\mathbb{N}_0$
on triangles. On curvilinear elements we use pull-backs of polynomials of degree $q$.
By duality, one establishes for $f\in H^s(\mathcal{T}_h)$ ($s\ge 0$) the bound
\begin{equation} \label{fh}
   \|f-\Pi_h^\Omega f\|_{-2,\Omega} \lesssim h^{\min\{2,q+1\}} \|f-\Pi_h^\Omega f\|_\Omega
   \lesssim h^{\min\{2,q+1\}+\min\{s,q+1\}} \|f\|_{s,\mathcal{T}_h}.
\end{equation}

A combination of Theorems~\ref{thm_NodInt},~\ref{thm_BestApp},~\ref{thm_BEM}, Corollary~\ref{cor_Pihp}, and estimate \eqref{fh}, proves the following convergence result.

\begin{cor} \label{cor_BEM}
  Under the notation and assumptions of Theorem~\ref{thm_BEM},
  suppose that $\tr{g}_h = I_h\tr{g}$ or $\tr{g}_h = \Pi_h\tr{g}$,
  $f_h:=\Pi_h^\Omega f$,
  that $\tr{g}$ as well as the solution $\tr{m}$ of~\eqref{BIE_saddle}
  satisfy the regularity assumptions of Theorems~\ref{thm_NodInt} and~\ref{thm_BestApp},
  and that $f\in H^s(\Omega)$ for $s\ge 0$. Then,
  \begin{align*}
    \|\tr{m}-\tr{m}_h\|_{-3/2,-1/2,\Gamma} = \mathcal{O}(h^{p+3/2})
    + \mathcal{O}(h^{\min\{2,q+1\}+\min\{s,q+1\}}).
  \end{align*}
  This gives
  \begin{align*}
    \|\tr{m}-\tr{m}_h\|_{-3/2,-1/2,\Gamma} = \mathcal{O}(h^{p+3/2})
    \quad\text{for}\quad \begin{cases}
                          q\ge 1\ \text{or}\ s\ge 1/2 & \text{if}\ p=0,\\
                          q\ge \max\{1,p-3/2\}\ \text{and}\ s\ge p-1/2 & \text{if}\ p\ge 1.
                       \end{cases}    
  \end{align*}
\end{cor}

\section{Numerical results} \label{sec_num}

We consider the following pairs of domains (circle, square, pacman)
and prescribed smooth solutions $u$ of~\eqref{model},
\begin{alignat*}{2}
	\Omega &:= \big\{x\in \R^2;\; |x|<0.1\big\}, &&u(x) := (x_1^2 + x_2^2) (x_1^2 + 2x_1 x_2 - x_2^2),
	\\
	\Omega &:= (-0.1,0.1)^2, &&u(x) := \cos\big(2\pi(x_1 + x_2)\big)
	\\
	\Omega &:= \big\{r\,(\cos\theta,\sin\theta);\; r\in [0,0.1), \theta \in (-\tfrac78\pi,\tfrac78\pi)\big\},\qquad &&u(x) := r^4 \cos(2\theta); 
\end{alignat*}
see Figures~\ref{fig:circle}--\ref{fig:pacman} for the respective geometry along with the considered initial mesh of the boundary and the interior (in case of the square). 
In each case, the corresponding boundary $\Gamma$ is parametrized with respect to the arc length and positive orientation. 
Note that the first and third solution satisfy \eqref{model} with homogeneous source term $f=0$, while $f\neq 0$ for the second solution.
We follow Section~\ref{sec_BEM} to approximate the solution
$\tr{m}=\trdD{-\cC\Gradgrad u}$ of \eqref{BIE_saddle} for $\nu = 0$ on uniformly refined meshes --- using bisection for the boundary mesh and red refinement for the interior mesh --- by the solution $\tr{m}_h \in X_h$ of \eqref{BIE_saddle:disc} for $p\in\{0,1,2\}$,  $\tr{g}_h := \Pi_h\tr{g}$, and $f_h := \Pi_h^\Omega f$. 

Recall from Proposition~\ref{prop_Vposdef} that
$\tnorm{\cdot}^2 := \dual{\V(\cdot)}{(\cdot)}$ defines an equivalent \emph{energy} norm in
$H^{-3/2,-1/2}_*(\Gamma)$.  
Note that both $\tr{m}$ and $\tr{m}_h$ lie in $H^{-3/2,-1/2}_*(\Gamma)$.
In particular, the discretization error satisfies
\begin{align*}
	\| \tr{m} - \tr{m}_h \|_{-3/2,-1/2,\Gamma}^2 
	\eqsim \tnorm{\tr{m} - \tr{m}_h}^2
	= \dual{\V\tr{m}}{\tr{m}} - 2\dual{\V\tr{m}}{\tr{m}_h} + \dual{\V\tr{m}_h}{\tr{m}_h}.
\end{align*}
To compute the error, we replace $\tr{m} = (\trnn{m},\trsf{m},\trJ{m})$ by $(\pi_{h,p+1}\trnn{m},\pi_{h,p}\trsf{m},\trJ{m})$, where $\pi_{h,q}\colon$ $L_2(\Gamma) \to P_q(\cE_h)$ denotes the orthogonal projection for $q\in\mathbb{N}_0$.
Note that this approximation is of higher order, as
\begin{align*}
	&\| \tr{m} - (\pi_{h,p+1}\trnn{m},\pi_{h,p}\trsf{m},\trJ{m}) \|_{-3/2,-1/2,\Gamma}^2 
	\\ &\quad 
	= \sup_{\tr{v}\in H^{3/2,1/2}(\Gamma) \setminus\{0\}} \frac{\dual{(1-\pi_{h,p+1})\trsf{m}}{(1-\pi_{h,0})\trD{v}}_\Gamma - \dual{(1-\pi_{h,p}) \trnn{m}}{\trN{v}}_\Gamma }{\|\tr{v}\|_{3/2,1/2,\Gamma}}
	\\ &\quad
	\lesssim \sup_{\tr{v}\in H^{3/2,1/2}(\Gamma) \setminus\{0\}} \frac{h\|{(1-\pi_{h,p})\trsf{m}}\|_\Gamma \| \partial_t \trD{v}\|_\Gamma + \|(1-\pi_{h,p+1}) \trnn{m}\|_\Gamma \|\trN{v}\|_\Gamma }{\|\tr{v}\|_{3/2,1/2,\Gamma}}
	\\ &\quad
	\lesssim h\|{(1-\pi_{h,p})\trsf{m}}\|_\Gamma + \|(1-\pi_{h,p+1}) \trnn{m}\|_\Gamma 
	= \mathcal{O}(h^{p+2}).
\end{align*}
Figures~\eqref{fig:circle}--\eqref{fig:pacman} show the (approximate)
discretization errors measured in the energy norm.
We observe the expected convergence $\tnorm{\tr{m} - \tr{m}_h} = \mathcal{O}(h^{p+3/2}) = \mathcal{O}({\rm dofs}^{-p-3/2})$,
where ${\rm dofs}$ denotes the number of degrees of freedom, i.e., the dimension of $X_h$. 

\newcommand{\plotMyFigure}[4]{
\begin{figure}[ht]
\centering
\begin{tikzpicture}
\begin{axis}[width = 0.7\textwidth, xlabel={$x_1$}, ylabel={$x_2$}, axis equal,
xticklabel style={/pgf/number format/fixed}, yticklabel style={/pgf/number format/fixed}]
#3 
\end{axis}
\end{tikzpicture}

\medskip

\begin{tikzpicture}
\pgfplotstableread[col sep=comma]{data/geo-#1_sol-#2_p-0_q-0.csv}{\pzero}
\pgfplotstableread[col sep=comma]{data/geo-#1_sol-#2_p-1_q-1.csv}{\pone}
\pgfplotstableread[col sep=comma]{data/geo-#1_sol-#2_p-2_q-2.csv}{\ptwo}
\begin{loglogaxis}[width = 0.7\textwidth, xlabel={degrees of freedom}, ylabel={energy error}, xmajorgrids=true, ymajorgrids=true, 
legend style={legend pos=south west}]
\addplot[red, mark=*, very thick] table[x=dofs,y=errs]{\pzero};
\addplot[blue, mark=square, very thick] table[x=dofs,y=errs]{\pone};
\addplot[green, mark=x, very thick] table[x=dofs,y=errs]{\ptwo};

\pgfplotstablegetrowsof{\pzero}
\pgfplotstablegetelem{0}{dofs}\of\pzero
\let\firstdofzero\pgfplotsretval

\pgfplotstablegetrowsof{\pzero}
\pgfmathtruncatemacro{\lastrow}{\pgfplotsretval-1} 
\pgfplotstablegetelem{\lastrow}{dofs}\of\pzero
\let\lastdofzero\pgfplotsretval

\pgfplotstablegetrowsof{\pone}
\pgfplotstablegetelem{0}{dofs}\of\pone
\let\firstdofone\pgfplotsretval

\pgfplotstablegetrowsof{\pone}
\pgfmathtruncatemacro{\lastrow}{\pgfplotsretval-1}
\pgfplotstablegetelem{\lastrow}{dofs}\of\pone
\let\lastdofone\pgfplotsretval

\pgfplotstablegetrowsof{\ptwo}
\pgfplotstablegetelem{0}{dofs}\of\ptwo
\let\firstdoftwo\pgfplotsretval

\pgfplotstablegetrowsof{\ptwo}
\pgfmathtruncatemacro{\lastrow}{\pgfplotsretval-1}
\pgfplotstablegetelem{\lastrow}{dofs}\of\ptwo
\let\lastdoftwo\pgfplotsretval

#4 

\legend{
$p=0$,
$p=1$,
$p=2$,
}
\end{loglogaxis}
\end{tikzpicture}
\caption{Initial mesh (top) and convergence plot (bottom) for the #1.
}
\label{fig:#1}
\end{figure}
}

\plotMyFigure{circle}{quartic}{
\addplot[domain=0:360,samples=200,thick]({0.1*cos(x)}, {0.1*sin(x)});
\node[red, fill=red, circle, inner sep=2pt] at (axis cs:0.1,0) {};
\node[red, fill=red, circle, inner sep=2pt] at (axis cs:0,0.1) {};
\node[red, fill=red, circle, inner sep=2pt] at (axis cs:-0.1,0) {};
\node[red, fill=red, circle, inner sep=2pt] at (axis cs:0,-0.1) {};
}{
\addplot[black,dashed,domain=\firstdofzero:\lastdofzero] {0.3*x^(-3/2) };
\node at (axis cs:2*1e2,7*1e-4) [anchor=north west] {$\mathcal{O}({\rm dofs}^{-3/2})$};
\addplot[black,dashed,domain=\firstdofone:\lastdofone] {0.7*x^(-5/2) };
\node at (axis cs:6.3*1e2,1.5*1.0e-8) [anchor=north west] {$\mathcal{O}({\rm dofs}^{-5/2})$};
\addplot[black,dashed,domain=\firstdoftwo:\lastdoftwo] {10*x^(-7/2) };
\node at (axis cs:1e2,1.0e-8) [anchor=north west] {$\mathcal{O}({\rm dofs}^{-7/2})$};
}

\plotMyFigure{square}{cosine}{
\addplot[thick]coordinates {(-0.1,-0.1) (0.1,-0.1) (0.1,0.1) (-0.1,0.1) (-0.1,-0.1)};
\addplot[thick]coordinates {(-0.1,0.1) (0.1,-0.1)};
\node[red, fill=red, circle, inner sep=2pt] at (axis cs:0.1,0.1) {};
\node[red, fill=red, circle, inner sep=2pt] at (axis cs:-0.1,0.1) {};
\node[red, fill=red, circle, inner sep=2pt] at (axis cs:0.1,-0.1) {};
\node[red, fill=red, circle, inner sep=2pt] at (axis cs:-0.1,-0.1) {};
}{
\addplot[black,dashed,domain=\firstdofzero:\lastdofzero] {100*x^(-3/2) };
\node at (axis cs:2*1e2,3*1e-1) [anchor=north west] {$\mathcal{O}({\rm dofs}^{-3/2})$};
\addplot[black,dashed,domain=\firstdofone:\lastdofone] {100*x^(-5/2) };
\node at (axis cs:6*1e2,2*1.0e-6) [anchor=north west] {$\mathcal{O}({\rm dofs}^{-5/2})$};
\addplot[black,dashed,domain=\firstdoftwo:\lastdoftwo] {900*x^(-7/2) };
\node at (axis cs:1e2,9*1.0e-7) [anchor=north west] {$\mathcal{O}({\rm dofs}^{-7/2})$};
}

\plotMyFigure{pacman}{singular}{
\addplot[domain=0:157.5,samples=200,thick]({0.1*cos(x)}, {0.1*sin(x)});
\addplot[domain=202.5:360,samples=200,thick]({0.1*cos(x)}, {0.1*sin(x)});
\addplot[thick] coordinates {(0,0) ({0.1*cos(360-157.5)},{0.1*sin(360-157.5)})};
\addplot[thick] coordinates {(0,0) ({0.1*cos(360-157.5)},{0.1*sin(157.5)})};
\node[red, fill=red, circle, inner sep=2pt] at (axis cs:0,0) {};
\node[red, fill=red, circle, inner sep=2pt] at (axis cs:0.1,0) {};
\node[red, fill=red, circle, inner sep=2pt] at (axis cs:-0.09239,0.03827) {};
\node[red, fill=red, circle, inner sep=2pt] at (axis cs:-0.09239,-0.03827) {};
\node at (axis cs:-0.05,0) {$\pi/4$};
}{
\addplot[black,dashed,domain=\firstdofzero:\lastdofzero] {0.4*x^(-3/2) };
\node at (axis cs:2*1e2,5*1e-4) [anchor=north west] {$\mathcal{O}({\rm dofs}^{-3/2})$};
\addplot[black,dashed,domain=\firstdofone:\lastdofone] {3*x^(-5/2) };
\node at (axis cs:6.25*1e2,6*1.0e-8) [anchor=north west] {$\mathcal{O}({\rm dofs}^{-5/2})$};
\addplot[black,dashed,domain=\firstdoftwo:\lastdoftwo] {70*x^(-7/2) };
\node at (axis cs:1e2,6*1.0e-8) [anchor=north west] {$\mathcal{O}({\rm dofs}^{-7/2})$};
}

\appendix

\section{Numerical computation of singular integrals} \label{sec_app}
In this section, we describe how to numerically approximate the singular integrals occurring in the discrete variational problem~\eqref{BIE_saddle:disc}.
More precisely, we consider $\dual{\V\tr{m}_h}{\tr{q}_h}$, $\dual{\K\tr{u}_h}{\tr{q}_h}$, and $\dual{\NIO f_h}{\tr{q}_h}$ for $\tr{m}_h = (\trhnn{m},\trhsf{m},\trhJ{m}), \tr{q}_h = (\trhnn{q},\trhsf{q},\trhJ{q}) \in X_h$, $\tr{u}_h = (\trhD{u},\trhN{u}) \in Y_h$, and $f_h \in P_q(\mathcal{T}_h)$.
Throughout, we abbreviate $\cN:= \cN_h$ for $p=0$ and $\cN:=\cN(\Omega)$ for $p>0$.
We note that the employed Theorems~\ref{thm_BIO_V}, \ref{thm_BIO_K} hold accordingly for $\cN(\Omega)$ replaced by $\cN_h$ (or any other set that contains $\cN(\Omega)$).

\subsection{Single-layer operator} \label{sec_computeV}
We have that
\begin{align*}
	\dual{\V\tr{m}_h}{\tr{q}_h}
	= \dual{\Vo\tr{m}_h}{\trhsf{q}}_\Gamma - \dual{\Vt\tr{m}_h}{\trhnn{q}}_\Gamma + \sum_{x\in\cN} \Vo\tr{m}_h(x) \trhJ{q}(x).
\end{align*}
By Theorem~\ref{thm_BIO_V}, the first term is given as 
\begin{subequations}
\begin{align}
	&\dual{\Vo\tr{m}_h}{\trhsf{q}}_\Gamma  \label{eq:Vsfsf}
	=  \frac 1{8\pi} \int_\Gamma \int_\Gamma |x-y|^2\log|x-y| \, \trhsf{m}(y) \trhsf{q}(x)\ds_y \ds_x
	\\ &\qquad \label{eq:Vnnsf}
	 	+ \frac 1{8\pi} \int_\Gamma \int_\Gamma (2\log|x-y|+1){n}_y\cdot(x-y) \, \trhnn{m}(y) \trhsf{q}(x) \ds_y \ds_x
	\\ &\qquad \label{eq:VJsf}
		+ \frac 1{8\pi} \int_\Gamma \sum_{y\in\cN} |x-y|^2\log|x-y| \, \trhJ{m}(y) \trhsf{q}(x) \ds_x,
\end{align}
\end{subequations}
the second one as 
\begin{subequations}
\begin{align}
	&\dual{\Vt\tr{m}_h}{\trhnn{q}}_\Gamma \label{eq:Vsfnn}
	= \frac 1{8\pi}\int_\Gamma \int_\Gamma (2\log|x-y|+1){n}_x\cdot(x-y) \, \trhsf{m}(y) \trhnn{q}(x) \ds_y \ds_x
	\\ & \label{eq:Vnnnn}
		+ \frac 1{8\pi} \int_\Gamma \int_\Gamma \Bigl((2\log|x-y|+1){n}_x\cdot{n}_y + 2 \frac{{n}_x\cdot(x-y)\,{n}_y\cdot(x-y)}{|x-y|^2} \Bigr) \trhnn{m}(y) \trhnn{q}(x) \ds_y \ds_x 
	\\ & \label{eq:VJnn}
		+ \frac 1{8\pi} \int_\Gamma \sum_{y\in\cN} (2\log|x-y|+1){n}_x\cdot(x-y) \trhJ{m}(y) \trhnn{q}(x) \ds_x,
\end{align}
\end{subequations}
and the third one as
\begin{subequations}
\begin{align}
	&\sum_{x\in\cN} \Vo\tr{m}_h(x) \trhJ{q}(x) \label{eq:VsfJ}
	=  \frac 1{8\pi} \sum_{x\in\cN} \int_\Gamma |x-y|^2\log|x-y| \, \trhsf{m}(y) \trhJ{q}(x)\ds_y 
	\\ &\qquad \label{eq:VnnJ}
	 	+ \frac 1{8\pi} \sum_{x\in\cN} \int_\Gamma (2\log|x-y|+1){n}_y\cdot(x-y) \, \trhnn{m}(y) \trhJ{q}(x) \ds_y
	\\ &\qquad \label{eq:VJJ}
		+ \frac 1{8\pi} \sum_{x\in\cN} \sum_{y\in\cN} |x-y|^2\log|x-y| \, \trhJ{m}(y) \trhJ{q}(x).
\end{align}
\end{subequations}
Recalling that $0^2\log |0|^2 := 0$, the term~\eqref{eq:VJJ} can be trivially evaluated. 
By symmetry, \eqref{eq:Vsfnn}, \eqref{eq:VsfJ}, and \eqref{eq:VnnJ} can be reduced to \eqref{eq:Vnnsf}, \eqref{eq:VJsf}, and \eqref{eq:VJnn}, respectively. 
It remains to consider \eqref{eq:Vsfsf}, \eqref{eq:Vnnsf}, \eqref{eq:Vnnnn}, \eqref{eq:VJsf}, and \eqref{eq:VJnn}. 
We split all outer and inner integrals over $\Gamma$ in integrals over $E\in\cE_h$ and $E'\in\cE_h$. 

The integrand in~\eqref{eq:Vsfsf} is continuous, but higher-order derivatives have logarithmic singularities. 
If $E\cap E'$ is empty, we may use standard Gauss--Legendre quadrature (in the parameter domain). 
Otherwise, we first apply a Duffy transformation (in the parameter domain) along the point or line singularity, see, e.g., \cite{gps22}, and then apply the quadrature from~\cite{smith00}, which is exact for integrands of the form $s \mapsto f_1(s) + f_2(s) \log(s)$ for polynomials $f_1, f_2$. 

The same procedure can be applied to~\eqref{eq:Vnnsf} and the first summand in \eqref{eq:Vnnnn}. 
Note that the second summand in \eqref{eq:Vnnnn} is continuous, as 
\begin{align}\label{eq:n_diff}
	{n}_x\cdot(x-y) = \mathcal{O}(|x-y|^2) = {n}_y\cdot(x-y) \quad (|x-y|\to 0).
\end{align}
We can thus again use standard Gauss--Legendre quadrature, in combination with a Duffy transformation if $E\cap E'\neq \emptyset$. 

For \eqref{eq:VJsf} and \eqref{eq:VJnn}, we employ Gauss--Legendre quadrature if $E$ or $E'$ contains the considered vertex $y\in\cN$ or $x\in\cN$, respectively, and the quadrature from~\cite{smith00} else.

\subsection{Double-layer operator}
We have that
\begin{align*}
	\dual{\K\tr{u}_h}{\tr{q}_h} 
	= \dual{\Ko\tr{u}_h}{\trhsf{q}}_\Gamma - \dual{\Kt\tr{u}_h}{\trhnn{q}}_\Gamma + \sum_{x\in\cN} \Ko\tr{u}_h(x) \trhJ{q}(x).
\end{align*}
By Theorem~\ref{thm_BIO_K}, the first term is given as 
\begin{subequations}
\begin{align}
	&\dual{\Ko\tr{u}_h}{\trhsf{q}}_\Gamma \label{eq:K0sf}
	= - \frac 1{2\pi} \int_\Gamma \int_\Gamma \partial_{n_y}\log|x-y| \, \trhD{u}(y) \trhsf{q}(x)\ds_y \ds_x
	\\ &\qquad \label{eq:Ktsf}
		+ \frac {1-\nu}{4\pi} \int_\Gamma \int_\Gamma \frac{{t}_y\cdot(x-y)\, {n}_y\cdot(x-y)}{|x-y|^2} \, \partial_{t_y}\trhD{u}(y) \trhsf{q}(x) \ds_y \ds_x
	\\ &\qquad \label{eq:Knsf}
		+ \int_\Gamma \int_\Gamma \Bigl(\frac {1+3\nu}{8\pi} + \frac {1+\nu}{4\pi} \log|x-y| + \frac {1-\nu}{4\pi} \frac {\bigl({n}_y\cdot (x-y)\bigr)^2}{|x-y|^2} \Bigr) \, \trhN{u}(y) \trhsf{q}(x) \ds_y \ds_x
	\\ & \qquad \nonumber
		+ \int_\Gamma \bigl(\sigma(x)-\frac 12\bigr) \, \trhD{u}(x) \trhsf{q}(x) \ds_x,
\end{align}
\end{subequations}
(where the last term vanishes, because $\sigma = 1/2$ almost everywhere), 
the second one as 
\begin{subequations}
\begin{align} \label{eq:K0nn}
	&\dual{\Kt\tr{u}_h}{\trhnn{q}}_\Gamma 
	= - \frac 1{2\pi} \int_\Gamma \int_\Gamma \partial_{n_x}\partial_{n_y}\log|x-y| \, \bigl(\trhD{u}(y)-\trhD{u}(x)\bigr) \, \trhnn{q}(x) \ds_y \ds_x
	\\ &\qquad \label{eq:Ktnn}
		+ \frac {1-\nu}{4\pi} \int_\Gamma \int_\Gamma \partial_{n_x} \frac{{t}_y\cdot(x-y) {n}_y\cdot(x-y)}{|x-y|^2} \, \partial_{t_y} \trhD{u}(y) \trhnn{q}(x) \ds_y \ds_x 
	\\ &\qquad \label{eq:Knnn}
		+ \int_\Gamma \int_\Gamma \Bigl(\frac{1+\nu}{4\pi} \partial_{n_x}\log|x-y| + \frac {1-\nu}{4\pi} \partial_{n_x} \frac{({n}_y\cdot (x-y))^2}{|x-y|^2} \Bigr) \, \trhN{u}(y) \trhnn{q}(x) \ds_y \ds_x.
\end{align}
\end{subequations}
(where the last term can be trivially evaluated), and the third one as
\begin{subequations}
\begin{align}
	&\sum_{x\in\cN} \Ko\tr{u}_h(x) \trhJ{q}(x) \label{eq:K0J}
	= - \frac 1{2\pi} \sum_{x\in\cN} \int_\Gamma \partial_{n_y}\log|x-y| \, \trhD{u}(y) \trhJ{q}(x)\ds_y 
	\\ &\qquad \label{eq:KtJ}
		+ \frac {1-\nu}{4\pi} \sum_{x\in\cN} \int_\Gamma \frac{{t}_y\cdot(x-y)\, {n}_y\cdot(x-y)}{|x-y|^2} \partial_{t_y}\, \trhD{u}(y) \trhJ{q}(x) \ds_y 
	\\ &\qquad \label{eq:KnJ}
		+ \sum_{x\in\cN} \int_\Gamma \Bigl(\frac {1+3\nu}{8\pi} + \frac {1+\nu}{4\pi} \log|x-y| + \frac {1-\nu}{4\pi} \frac {\bigl({n}_y\cdot (x-y)\bigr)^2}{|x-y|^2} \Bigr) \, \trhN{u}(y) \trhJ{q}(x) \ds_y 
	\\ &\qquad \nonumber
		+ \sum_{x\in\cN} \bigl(\sigma(x)-\frac 12\bigr) \, \trhD{u}(x) \trhJ{q}(x). 
\end{align}
\end{subequations}
As before in Section~\ref{sec_computeV}, we split all outer and inner integrals over $\Gamma$ in integrals over $E\in\cE_h$ and $E'\in\cE_h$. 

The kernel of~\eqref{eq:K0sf} is exactly the kernel of the double-layer operator associated to the Laplacian. 
Owing to~\eqref{eq:n_diff}, we can again use standard Gauss--Legendre quadrature in combination with a Duffy transformation if $E\cap E'\neq \emptyset$, cf.~\cite{gps22}. 
The same procedure can be applied to~\eqref{eq:Ktsf}, the third term of~\eqref{eq:Knsf}, the first term of~\eqref{eq:Knnn}, and, with the explicit representation  
\begin{align*}
	\frac {1-\nu}{4\pi}  \partial_{n_x} \frac{({n}_y\cdot (x-y))^2}{|x-y|^2}
	= \frac {1-\nu}{2\pi} \Big(\frac{{n}_y\cdot (x-y) {n}_x \cdot {n}_y}{|x-y|^2} - \frac{({n}_y\cdot (x-y))^2 {n}_x \cdot (x-y)}{|x-y|^4}\Big),
\end{align*}
also for the second term of~\eqref{eq:Knnn}.

The kernel of the first term of~\eqref{eq:Knsf} is constant, so that Gauss--Legendre quadrature can be used.
The kernel of the second term of~\eqref{eq:Knsf} has a logarithmic singularity, and can thus be computed as in Section~\ref{sec_computeV}. 

The integral kernel of~\eqref{eq:K0nn} is minus the kernel of the hypersingular operator associated to the Laplacian. 
The integration-by-parts formula from \cite[Section~6.5]{Steinbach_08_NAM} thus yields the representation
\begin{subequations}
\begin{align}
	&- \frac 1{2\pi} \int_\Gamma \int_\Gamma \partial_{n_x}\partial_{n_y}\log|x-y| \, \bigl(\trhD{u}(y)-\trhD{u}(x)\bigr) \, \trhnn{q}(x) \ds_y \ds_x \nonumber
	\\ &\qquad \label{eq:K0nn_a}
	= \frac 1{2\pi} \sum_{E\in\cE_h} \sum_{E'\in\cE_h} \int_E \int_{E'} \log|x-y|  \, \partial_{t_y} \trhD{u}(y) \, \partial_{t_x} \trhnn{q}(x) \ds_y \ds_x 
	\\ &\qquad\quad \label{eq:K0nn_b}
		+ \frac 1{2\pi} \sum_{x\in\cN_h} \sum_{E'\in\cE_h} \int_{E'} \log|x-y|  \, \partial_{t_y} \trhD{u}(y) \, [\trhnn{q}(x)] \ds_y.
\end{align}
\end{subequations}
The integral kernels of~\eqref{eq:K0nn_a}--\eqref{eq:K0nn_b} are both minus the integral kernel of the single-layer operator associated to the Laplacian. 
The terms in \eqref{eq:K0nn_a} can thus be computed as before in Section~\ref{sec_computeV} by using standard Gauss--Legendre quadrature if $E\cap E' = \emptyset$ and a Duffy transformation in combination with the quadrature from~\cite{smith00} if $E\cap E' \neq \emptyset$. 
Similarly, the terms in \eqref{eq:K0nn_b} can be computed by using Gauss--Legendre quadrature if $x\not\in E'$ and the quadrature from~\cite{smith00} else. 

For~\eqref{eq:Ktnn}, we recall from~\eqref{calc} that the kernel is given as 
\begin{align*}
	\frac {1-\nu}{4\pi}  \partial_{n_x} \frac{{t}_y\cdot(x-y) {n}_y\cdot(x-y)}{|x-y|^2}
	= \partial_{n_x} \big({t}_y\cdot\cC\Gradgrad G(x-y) {n}_y\big)
	= (1-\nu) \big(\partial_{t_y} \partial_{n_x} \grad_y G(x-y)\big) \cdot{n}_y.
\end{align*}
Hence, the integration-by-parts formula from \cite[Section~6.5]{Steinbach_08_NAM} yields
\begin{subequations}
\begin{align} 
	&\frac {1-\nu}{4\pi} \int_\Gamma \int_\Gamma \partial_{n_x} \frac{{t}_y\cdot(x-y) {n}_y\cdot(x-y)}{|x-y|^2} \, \partial_{t_y} \trhD{u}(y) \trhnn{q}(x) \ds_y \ds_x \nonumber
	\\ &\quad \label{eq:Ktnn_a}
	= -(1-\nu) \sum_{E\in\cE_h} \sum_{E'\in\cE_h} \int_E \int_{E'} \partial_{n_x} \grad_y G(x-y) \cdot \partial_{t_y}\big(\partial_{t_y} \trhD{u}(y) {n}_y\big) \trhnn{q}(x) \ds_y \ds_x 
	\\ &\qquad \label{eq:Ktnn_b}
		- (1-\nu) \sum_{E\in\cE_h} \sum_{y\in\cN_h} \int_E \partial_{n_x} \grad_y G(x-y) \cdot \big[\partial_{t_y} \trhD{u}(y) {n}_y \big] \trhnn{q}(x) \ds_y \ds_x, 
\end{align}
where, using again~\eqref{calc}, 
\begin{align} \label{eq:Ktnn_c}
	\partial_{n_x}\grad_y G(x-y) 
	= -\Gradgrad G(x-y) {n}_x
	= -\frac 1{8\pi}(2\log|(x-y)|+1){n}_x - \frac 1{4\pi} \frac{{n}_x\cdot (x-y)(x-y)}{|x-y|^2}.
\end{align}
\end{subequations}
The first term in~\eqref{eq:Ktnn_a} (from~\eqref{eq:Ktnn_c}) has a logarithmic singularity and can thus be computed as before in Section~\ref{sec_computeV} by using standard Gauss--Legendre quadrature if $E\cap E' = \emptyset$ and a Duffy transformation in combination with the quadrature from~\cite{smith00} if $E\cap E' \neq \emptyset$. 
Owing to \eqref{eq:n_diff}, the second term in \eqref{eq:Ktnn_a} (from~\eqref{eq:Ktnn_c}) can be computed by using Gauss--Legendre quadrature if $E\cap E' = \emptyset$ in combination with a Duffy transformation if $E\cap E' \neq \emptyset$. 

For~\eqref{eq:K0J}, \eqref{eq:KtJ}, the third term in~\eqref{eq:KnJ},
and the second term in~\eqref{eq:Ktnn_b} (from~\eqref{eq:Ktnn_c}), we employ Gauss--Legendre quadrature, which is justified by~\eqref{eq:n_diff}. 
We further employ Gauss--Legendre quadrature for the first term in \eqref{eq:KnJ} with smooth (constant) kernel as well as for the second term in~\eqref{eq:KnJ} and the first term in \eqref{eq:Ktnn_b} (from~\eqref{eq:Ktnn_c}) with logarithmic kernel if $x \not \in E'$, and the quadrature from~\cite{smith00} if $x \in E'$.

\subsection{Newton potential} \label{sec:compute_N}
We have that
\begin{align*}
	\dual{\NIO f_h}{\tr{q}_h} 
	= \dual{\No f_h}{\trhsf{q}}_\Gamma - \dual{\Nt f_h}{\trhnn{q}}_\Gamma + \sum_{x\in\cN} \No f_h(x) \trhJ{q}(x),
\end{align*}
and Theorem~\ref{thm_NIO} reveals the representation
\begin{subequations}
\begin{align}
	&\dual{\No f_h}{\trhsf{q}}_\Gamma  \label{eq:Nsf}
	= \frac 1{8\pi} \int_\Gamma \int_\Omega  |x-y|^2\log|x-y| \, f_h(y) \trhsf{q}(x)\dy \ds_x
	\\ &\qquad \label{eq:Nnn}
	- \frac 1{8\pi}\int_\Gamma \int_\Omega (2\log|x-y|+1){n}_x\cdot({x}-{y}) \, f_h(y) \trhnn{q}(x)\dy \ds_x
	\\ &\qquad \label{eq:NJ}
	+ \frac 1{8\pi} \sum_{x\in\cN} \int_\Omega  |x-y|^2\log|x-y| \, f_h(y) \trhJ{q}(x) \dy. 
\end{align}
\end{subequations}

Applying first a Duffy transformation to map triangles $T\in\mathcal{T}_h$ onto the unit square $[0,1]^2$, the two-dimensional integral~\eqref{eq:NJ} can be treated as the term~\eqref{eq:Vsfsf}. 
Note that \cite{gps22} only covers one- and two-dimensional integrals with logarithmic singularity. 

The integrand of the second term in the three-dimensional integral \eqref{eq:Nnn} is smooth so that standard quadrature can be used. 
The remaining three-dimensional integrals in \eqref{eq:Nsf} and \eqref{eq:Nnn} have both logarithmic kernels and can be treated similarly: 
We abbreviate their integrands by $F(x,y)$ and write 
\begin{align*}
	\int_\Gamma \int_\Omega F(x,y) \dy \ds_x 
	= \sum_{E\in\cE_h} \sum_{T\in\mathcal{T}_h} \int_E \int_T F(x,y) \dy \ds_x.
\end{align*}
Let $E\in\cE_h$ and $T\in\mathcal{T}_h$. To ease presentation, we assume that $T = {\rm conv}\{(0,0)^\top, (1,0)^\top, (0,1)^\top\}$ is the reference triangle and that $E$ is a flat line. 
We further assume without loss of generality that the intersection between $T$ and $E$ is either empty, a vertex of $T$, or the edge $E$ itself. 
If $T\cap E = \emptyset$, the integrand is smooth so that again standard quadrature can be used. 
In the latter two cases, we may assume that $T\cap E = \{(1,0)^\top\}$ and $T\cap E = {\rm conv}\{(1,0)^\top,(0,1)^\top\}$, respectively. 
If $E\cap T = \{(1,0)^\top\}$, we parametrize $E$ via the affine mapping $\gamma_E:[0,1]\to E$ with $\gamma_E(0) = (1,0)^\top$. 
We apply the Duffy transformation $(s,y_1,y_2) \mapsto (s,1-y_1,y_1 y_2)$ with Jacobi determinant $-y_1$, split the resulting cube along the diagonal $\{(s,y_1,y_2)\in[0,1]^3;\; s=y_1\}$, apply the reflection $(s,y_1,y_2)\mapsto (y_1,s,y_2)$ on the prism $\{(s,y_1,y_2)\in [0,1]^3;\; s \le y_1\}$, and apply for both resulting integrals the Duffy transformation $(s,y_1,y_2) \mapsto (s,s y_1,y_2)$ with Jacobi determinant $s$ to ultimately obtain
\begin{align*}
	\int_E \int_T F(x,y) \dy \ds_x 
	&= |E| \int_{[0,1]^3} F\big(\gamma_E(s),(1-s y_1,s y_1 y_2)^\top\big) s^2 y_1 \dy \ds 
	\\ & \quad
	+ |E| \int_{[0,1]^3} F\big(\gamma_E(s y_1),(1-s, s y_2)^\top\big) s^2 \dy \ds,
\end{align*}
which can be accurately approximated with the quadrature from~\cite{smith00} in each direction. 
If $T\cap E = E$, we parametrize $E$ via the affine mapping $\gamma_E:[0,1]\to E$ with $\gamma_E(0) = (1,0)^\top$. 
We first apply the same transformations as for the case $E\cap T = \{(1,0)^\top\}$, and then split the cube along the diagonal $\{(s,y_1,y_2)\in[0,1]^3;\; y_1=y_2\}$, apply the reflection $(s,y_1,y_2)\mapsto (s,y_2,y_1)$ on the prism $\{(s,y_1,y_2)\in [0,1]^3;\; y_1 \le y_2\}$, and apply for all four resulting integrals the Duffy transformation $(s,y_1,y_2) \mapsto (s,1+(y_1 -1)y_2,1-y_2)$ with Jacobi determinant $y_2$  to obtain with $w(s,y_1,y_2):=s\big(1+y_2(y_1-1)\big)$
\begin{align*}
	\int_E \int_T F(x,y) \dy \ds_x 
	&= |E| \int_{[0,1]^3} F\big(\gamma_E(s),(1-w(s,y_1,y_2),(1-y_2)w(s,y_1,y_2))^\top\big) s w(s,y_1,y_2) y_2 \dy \ds 
	\\ & \quad
	+ |E| \int_{[0,1]^3} F\big(\gamma_E(w(s,y_1,y_2)),1-s, s(1-y_2))^\top\big) s^2 y_2 \dy \ds 
	\\ & \quad
	+ |E| \int_{[0,1]^3} F\big(\gamma_E(s),(1-s(1-y_2),(1-y_2) w(s,y_1,y_2))^\top\big) s^2 (1-y_2) y_2 \dy \ds 
	\\ & \quad
	+ |E| \int_{[0,1]^3} F\big(\gamma_E(s(1-y_2)),(1-s,w(s,y_1,y_2))^\top\big) s^2 y_2 \dy \ds,
\end{align*}
which can again be accurately approximated with the quadrature from~\cite{smith00} in each direction.

\bibliographystyle{siam}
\bibliography{references}

\begin{thebibliography}{10}

\bibitem{Beskos_91_BEA}
{\sc D.~E. Beskos}, ed., {\em Boundary Element Analysis of Plates and Shells},
  Springer Series in Computational Mechanics, Springer Berlin, Heidelberg,
  1991.

\bibitem{Bezine_78_BIF}
{\sc G.~B\'ezine}, {\em Boundary integral formulation for plate flexure with
  arbitrary boundary conditions}, Mech. Res. Comm., 5 (1978), pp.~197--206.

\bibitem{BlumR_80_BVP}
{\sc H.~Blum and R.~Rannacher}, {\em On the boundary value problem of the
  biharmonic operator on domains with angular corners}, Math. Methods Appl.
  Sci., 2 (1980), pp.~556--581.

\bibitem{Bourlard_88_PDB}
{\sc M.~Bourlard}, {\em Probl\`eme de {D}irichlet pour le bilaplacien dans un
  polygone: r\'{e}solution par \'{e}l\'{e}ments finis fronti\`eres
  raffin\'{e}s}, C. R. Acad. Sci. Paris S\'{e}r. I Math., 306 (1988),
  pp.~461--466.

\bibitem{CakoniHW_05_BIE}
{\sc F.~Cakoni, G.~C. Hsiao, and W.~L. Wendland}, {\em On the boundary integral
  equation method for a mixed boundary value problem of the biharmonic
  equation}, Complex Var. Theory Appl., 50 (2005), pp.~681--696.

\bibitem{ChristiansenH_78_IPI}
{\sc S.~Christiansen and P.~Hougaard}, {\em An investigation of a pair of
  integral equations for the biharmonic problem}, J. Inst. Math. Appl., 22
  (1978), pp.~15--27.

\bibitem{Ciarlet_78}
{\sc P.~G. Ciarlet}, {\em Interpolation error estimates for the reduced
  {H}sieh-{C}lough-{T}ocher triangle}, Math. Comp., 32 (1978), pp.~335--344.

\bibitem{Costabel_88_BIO}
{\sc M.~Costabel}, {\em Boundary integral operators on {Lipschitz} domains:
  Elementary results}, SIAM J. Math. Anal., 19 (1988), pp.~613--626.

\bibitem{CostabelD_96_IBS}
{\sc M.~Costabel and M.~Dauge}, {\em Invertibility of the biharmonic single
  layer potential operator}, Integral Equations Operator Theory, 24 (1996),
  pp.~46--67.

\bibitem{CostabelS_89_BEA}
{\sc M.~Costabel and J.~Saranen}, {\em Boundary element analysis of a direct
  method for the biharmonic {D}irichlet problem}, in The {G}ohberg anniversary
  collection, {V}ol.\ {II} ({C}algary, {AB}, 1988), vol.~41 of Oper. Theory
  Adv. Appl., Birkh\"auser, Basel, 1989, pp.~77--95.

\bibitem{CostabelS_83_NDD}
{\sc M.~Costabel and E.~P. Stephan}, {\em The normal derivative of the double
  layer potential on polygons and {Galerkin} approximation}, Appl. Anal., 16
  (1983), pp.~205--228.

\bibitem{CostabelS_85_BIE}
\leavevmode\vrule height 2pt depth -1.6pt width 23pt, {\em Boundary integral
  equations for mixed boundary value problems in polygonal domains and
  {Galerkin} approximations}, in Mathematical Models and Methods in Mechanics,
  W.~Fiszdon and K.~Wilmanski, eds., Warsaw, 1985, Banach Centre Publ. {\bf
  15}, pp.~175--251.

\bibitem{CostabelSW_83_BIE}
{\sc M.~Costabel, E.~P. Stephan, and W.~L. Wendland}, {\em On boundary integral
  equations of the first kind for the bi-{Laplacian} in a polygonal domain},
  Annali Scuola Normale Superiore Pisa, Ser. IV, 10 (1983), pp.~197--241.

\bibitem{DeCosterNS_15_SBD}
{\sc C.~De~Coster, S.~Nicaise, and G.~Sweers}, {\em Solving the biharmonic
  {D}irichlet problem on domains with corners}, Math. Nachr., 288 (2015),
  pp.~854--871.

\bibitem{ErnGuermondI}
{\sc A.~Ern and J.-L. Guermond}, {\em Finite elements {I}---{A}pproximation and
  interpolation}, vol.~72 of Texts in Applied Mathematics, Springer, Cham,
  2021.

\bibitem{fuglede81}
{\sc B.~Fuglede}, {\em On a direct method of integral equations for solving the
  biharmonic {D}irichlet problem}, Z. Angew. Math. Mech., 61 (1981),
  pp.~449--459.

\bibitem{FuehrerHH_23_DMQ}
{\sc T.~F{\"u}hrer, P.~Herrera, and N.~Heuer}, {\em A {DPG} method for the
  quad-curl problem}, Comput. Methods Appl. Math., 149 (2023), pp.~221--238.

\bibitem{FuehrerHN_19_UFK}
{\sc T.~F{\"u}hrer, N.~Heuer, and A.~H. Niemi}, {\em An ultraweak formulation
  of the {Kirchhoff}--{Love} plate bending model and {DPG} approximation},
  Math. Comp., 88 (2019), pp.~1587--1619.

\bibitem{gps22}
{\sc G.~Gantner, D.~Praetorius, and S.~Schimanko}, {\em Stable implementation
  of adaptive {IGABEM} in 2{D} in {MATLAB}}, Comput. Methods Appl. Math., 22
  (2022), pp.~563--590.

\bibitem{GiroireN_78_NSE}
{\sc J.~Giroire and J.-C. N\'ed\'elec}, {\em Numerical solution of an exterior
  {N}eumann problem using a double layer potential}, Math. Comp., 32 (1978),
  pp.~973--990.

\bibitem{GiroireN_95_NSB}
{\sc J.~Giroire and J.-C. N\'{e}d\'{e}lec}, {\em A new system of boundary
  integral equations for plates with free edges}, Math. Methods Appl. Sci., 18
  (1995), pp.~755--772.

\bibitem{GuoShuM_86_BEM}
{\sc S.~Guo-Shu and S.~Mukherjee}, {\em Boundary element analysis of bending of
  elastic plates of arbitrary shape with general boundary conditions}, Eng.
  Anal., 3 (1986), pp.~36--44.

\bibitem{GwinnerS_18_ABE}
{\sc J.~Gwinner and E.~P. Stephan}, {\em Advanced boundary element methods},
  vol.~52 of Springer Series in Computational Mathematics, Springer, Cham,
  2018.
\newblock Treatment of boundary value, transmission and contact problems.

\bibitem{HappelB_65_LRN}
{\sc J.~Happel and H.~Brenner}, {\em Low {R}eynolds number hydrodynamics with
  special applications to particulate media}, Prentice-Hall, Inc., Englewood
  Cliffs, NJ, 1965.

\bibitem{HartmannZ_86_DBE}
{\sc F.~Hartmann and R.~Zotemantel}, {\em The direct boundary element method in
  plate bending}, Int. J. Numer. Methods Eng., 23 (1986), pp.~2049--2069.

\bibitem{Heuer_26_GMP}
{\sc N.~Heuer}, {\em Generalized mixed and primal hybrid methods with
  applications to plate bending}, Numer. Math., 158 (2026), pp.~491--531.

\bibitem{HsiaoMC_73_SBV}
{\sc G.~C. Hsiao and R.~C. MacCamy}, {\em Solution of boundary value problems
  by integral equations of the first kind}, SIAM Rev., 15 (1973), pp.~687--705.

\bibitem{HsiaoW_77_FEM}
{\sc G.~C. Hsiao and W.~L. Wendland}, {\em A finite element method for some
  integral equations of the first kind}, J. Math. Anal. Appl., 58 (1977),
  pp.~449--481.

\bibitem{HsiaoW_21_BIE}
{\sc G.~C. Hsiao and W.~L. Wendland}, {\em Boundary integral equations},
  vol.~164 of Applied Mathematical Sciences, Springer, Cham, second~ed., 2021.

\bibitem{Jakovlev_61_BPF}
{\sc G.~N. Jakovlev}, {\em Boundary properties of functions of the class
  {$W\sb{p}{}\sp{(l)}$} in regions with corners}, Dokl. Akad. Nauk SSSR, 140
  (1961), pp.~73--76.

\bibitem{McLean_00_SES}
{\sc W.~McLean}, {\em Strongly Elliptic Systems and Boundary Integral
  Equations}, Cambridge University Press, 2000.

\bibitem{MitreaM_13_BVP}
{\sc I.~Mitrea and M.~Mitrea}, {\em Boundary value problems and integral
  operators for the bi-{L}aplacian in non-smooth domains}, Atti Accad. Naz.
  Lincei Rend. Lincei Mat. Appl., 24 (2013), pp.~329--383.

\bibitem{Muskhelishvili_53_SBP}
{\sc N.~I. Muskhelishvili}, {\em Some basic problems of the mathematical theory
  of elasticity. {F}undamental equations, plane theory of elasticity, torsion
  and bending}, P. Noordhoff Ltd., Groningen, 1953.
\newblock Translated by J. R. M. Radok.

\bibitem{NedelecP_73_MVE}
{\sc J.-C. N\'ed\'elec and J.~Planchard}, {\em Une m\'ethode variationnelle
  d'\'el\'ements finis pour la r\'esolution num\'erique d'un probl\`eme
  ext\'erieur dans {$R\sp{3}$}}, Rev. Fran\c caise Automat. Informat. Recherche
  Op\'erationnelle S\'er. Rouge, 7 (1973), pp.~105--129.

\bibitem{pb89}
{\sc C.~P. Providakis and D.~E. Beskos}, {\em Free and forced vibrations of
  plates by boundary elements}, Comput. Methods Appl. Mech. Engrg., 74 (1989),
  pp.~231--250.

\bibitem{SauterS_11_BEM}
{\sc S.~A. Sauter and C.~Schwab}, {\em Boundary element methods}, vol.~39 of
  Springer Series in Computational Mathematics, Springer-Verlag, Berlin, 2011.
\newblock Translated and expanded from the 2004 German original.

\bibitem{Schmidt_01_BIO}
{\sc G.~Schmidt}, {\em Boundary integral operators for plate bending in domains
  with corners}, Z. Anal. Anwendungen, 20 (2001), pp.~131--154.

\bibitem{SchmidtK_99_BIE}
{\sc G.~Schmidt and B.~N. Khoromskij}, {\em Boundary integral equations for the
  biharmonic {D}irichlet problem on nonsmooth domains}, J. Integral Equations
  Appl., 11 (1999), pp.~217--253.

\bibitem{smith00}
{\sc R.~Smith}, {\em {Direct Gauss quadrature formulae for logarithmic
  singularities on isoparametric elements}}, {Eng. Anal. Bound. Elem.}, 24
  (2000), pp.~161--167.

\bibitem{Steinbach_08_NAM}
{\sc O.~Steinbach}, {\em Numerical approximation methods for elliptic boundary
  value problems}, Springer, New York, 2008.
\newblock Finite and boundary elements, Translated from the 2003 German
  original.

\bibitem{Stephan_79_CMF}
{\sc E.~Stephan}, {\em Conform and mixed finite element schemes for the
  {Dirichlet} problem for the bi-{Laplacian} in plane domains with corners},
  Math. Methods Appl. Sci., 1 (1979), pp.~354--382.

\bibitem{morris79}
{\sc M.~Stern}, {\em A general boundary integral formulation for the numerical
  solution of plate bending problems}, Int. J. Solids Struct., 15 (1979),
  pp.~769--782.

\bibitem{TimoshenkoWK_59_TPS}
{\sc S.~Timoshenko and S.~Woinoksky-Krieger}, {\em Theory of Plates and
  Shells}, McGraw-Hill, 1959.

\bibitem{VentselK_01_TPS}
{\sc E.~Ventsel and T.~Krauthammer}, {\em Thin Plates and Shells}, CRC Press,
  New York, 2001.

\bibitem{WendlandSH_79_IEM}
{\sc W.~L. Wendland, E.~Stephan, and G.~C. Hsiao}, {\em On the integral
  equation method for the plane mixed boundary value problem of the
  {Laplacian}}, Math. Methods Appl. Sci., 1 (1979), pp.~265--321.

\end{thebibliography}
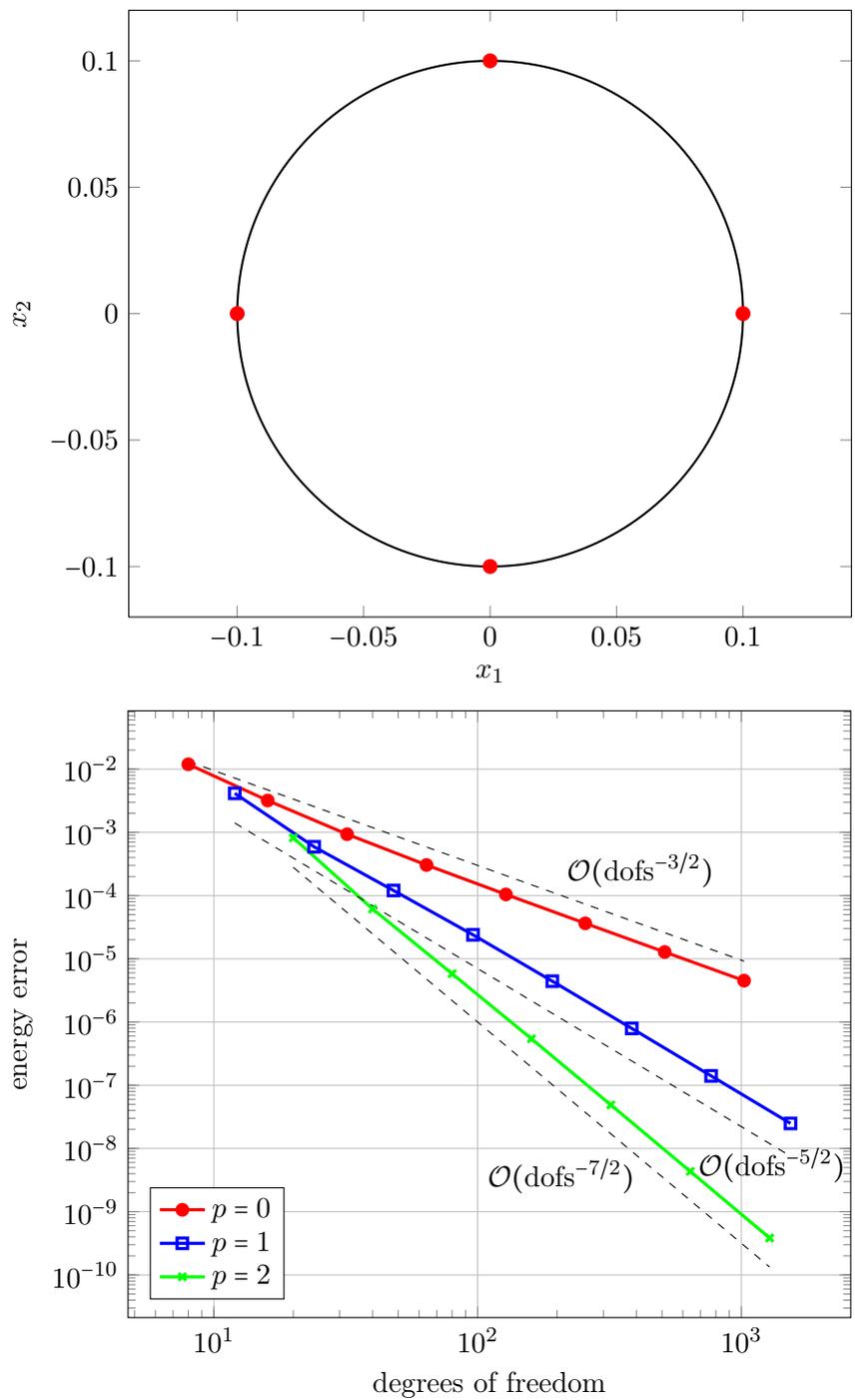
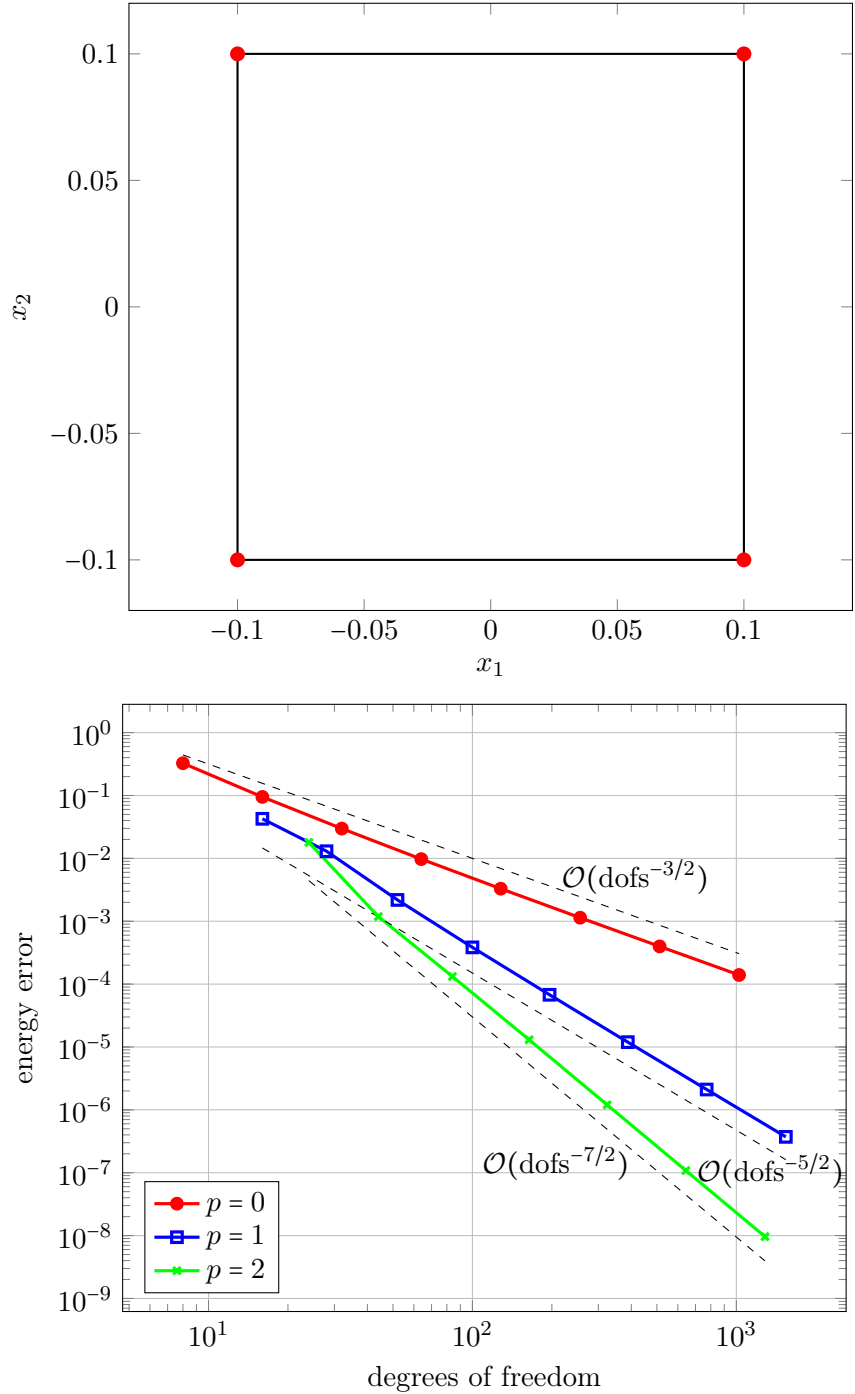
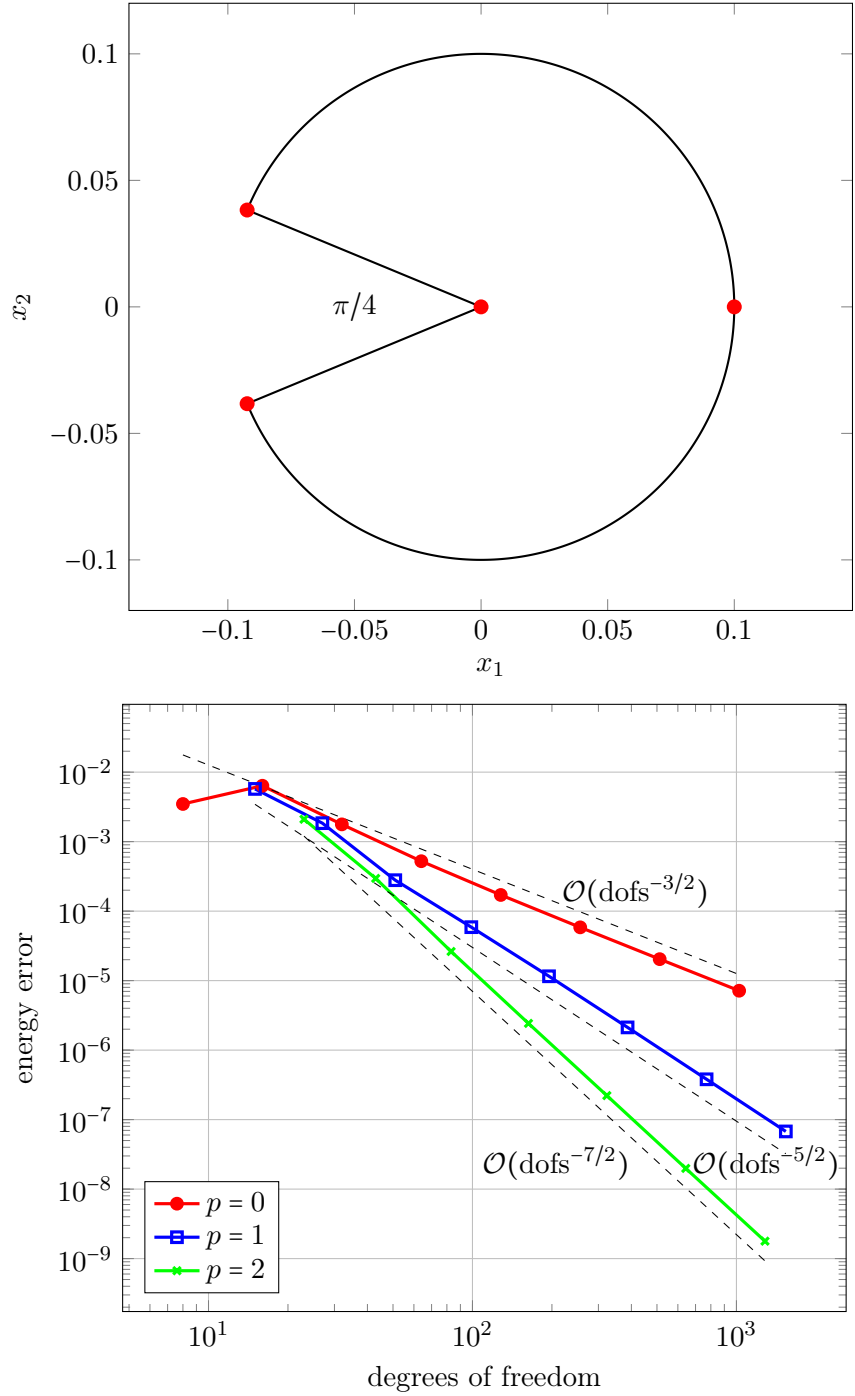
\end{document}